\documentclass[a4paper,10pt]{scrartcl}

\usepackage{latexsym} \usepackage{mathrsfs}

\usepackage{hyperref}

\usepackage{amsmath} \usepackage{amssymb} \usepackage{amsfonts}
\usepackage{amsthm} \usepackage{amsxtra}

\usepackage{epsfig} \input{xy} \xyoption{all}

\usepackage{nicefrac}

\newtheorem{theorem}{Theorem}[section]
\newtheorem{proposition}[theorem]{Proposition}
\newtheorem{lemma}[theorem]{Lemma}
\newtheorem{corollary}[theorem]{Corollary}
\newtheorem{question}[theorem]{Question}

\newtheorem*{maintheorem}{Theorem A}
\newtheorem*{maincorollary-I}{Corolary B}
\newtheorem*{maincorollary-II}{Corollary C}

\theoremstyle{definition} \newtheorem{definition}[theorem]{Definition}

\theoremstyle{remark} 

\numberwithin{equation}{section}

 \DeclareMathOperator{\Per}{Per}
\DeclareMathOperator{\Fix}{Fix} 
\DeclareMathOperator{\GL}{GL} 
 
\DeclareMathOperator{\cl}{cl} 

\DeclareMathOperator{\supp}{supp} 

\newcommand{\Hom}{\mathrm{Hom}} \newcommand{\Leb}{\mathrm{Leb}}
\newcommand{\Diff}[2]{\mathrm{Diff}_{#1}^{#2}}

\newcommand{\bb}[1]{\mathbb{#1}} \newcommand{\scr}[1]{\mathscr{#1}}
 \newcommand{\Lie}{\mathcal{L}}
\newcommand{\W}{\mathcal{W}} 
\newcommand{\T}{\mathbb{T}} \newcommand{\Ta}{\mathcal{T}}
\newcommand{\U}{\mathcal{U}} \newcommand{\A}{\scr{A}}
 \newcommand{\dd}{\:\mathrm{d}}
\newcommand{\Dis}[1]{\mathcal{D}^\prime_{#1}}
\newcommand{\Bs}{\mathcal{S}} \newcommand{\R}{\mathbb{R}}
\newcommand{\Z}{\mathbb{Z}} \newcommand{\Q}{\mathbb{Q}}
\newcommand{\Var}{\mathrm{Var}} \newcommand{\ie}{i.e.\ }
\newcommand{\eg}{e.g.\ }

\begin{document}

\title{Cohomological equations and invariant distributions for minimal
  circle diffeomorphisms} \author{Artur Avila and Alejandro Kocsard}
\maketitle

\begin{abstract}
  \textbf{Abstract:} Given any smooth circle diffeomorphism with
  irrational rotation number, we show that its invariant probability
  measure is the only invariant distribution (up to multiplication by
  a real constant).  As a consequence of this, we show that the space
  of real $C^\infty$-coboundaries of such a diffeomorphism is closed
  in $C^\infty(\T)$ if and only if its rotation number is Diophantine.
\end{abstract}

\section{Introduction}
\label{sec:intro}

Cohomological equations appear very frequently in different contexts
in dynamical systems. In fact, many problems, specially those
concerning with certain forms of rigidity and stability, can be
reduced to analyze the existence of solutions (in certain regularity
classes) of some cohomological equations (see~\cite{katok-robinson,
  katok} for general reference, \cite{GhysLinvariant,
  HurderDynamGodbillon-Vey} for applications to the study of
foliations and \cite{GhysGroupsActingCircle, NavasActionKazhdan} for
cohomological aspects of group actions on the circle).

In the case the dynamics is given by a diffeomorphism $f$ on a
manifold $M$, the most basic cohomological equation (and the only kind
we shall consider from now on) is a first order linear difference
equation
\begin{equation}
  \label{eq:CE-def}
  uf-u=\phi,
\end{equation}
where $\phi\colon M\to\R$ is given and $u\colon M\to\R$ is the unknown
of the problem.

In this work we shall mainly concern with cohomological equations in
the smooth category. In fact, most of the time we will assume that the
data of the equations, \ie the diffeomorphism $f$ and the function
$\phi$ in~(\ref{eq:CE-def}), are $C^\infty$ and be interested in the
existence of smooth solutions.

By analogy with the cohomology of groups, we can consider the function
$\phi$ in (\ref{eq:CE-def}) as being a smooth \emph{cocycle} over $f$,
and we say that $\phi$ is a (smooth) \emph{coboundary} whenever
(\ref{eq:CE-def}) admits a $C^\infty$ solution. Of course, this leads
us to define the \emph{first cohomology space} $H^1(f,C^\infty(M))$
(see \S\ref{sec:cocy-coub-distrib} for details).

In general these cohomology spaces could be rather ``wild'' (\eg its
natural topology is non-Hausdorff), and so it is rather hard to study
the structure of these spaces. However, we can distinguish two aspects
that appear as the fundamental characters in the analysis of
$H^1(f,C^\infty(M))$:
\begin{itemize}
\item[\textit{(i)}] the first one is the space of \emph{$f$-invariant
    distributions} in the sense of Schwartz (see
  \S\ref{sec:cocy-coub-distrib} for details);
\item[\textit{(ii)}] and the second one is the concept of
  \emph{cohomological stability} (see
  Definition~\ref{def:cohomological-stability}).
\end{itemize}

It is important to remark that, in general, the second problem is
considerably much harder than the first one. The work of Heafliger and
Banghe~\cite{haefliger} is a good testimony of this.

\subsection{Cohomological equations over quasi-periodic systems}
\label{sec:cohomo-eq-rotations}

Equations like~(\ref{eq:CE-def}) where $f=R_\alpha\colon\T^d\to\T^d$
is an ergodic rigid rotation on the $d$-torus appear as ``linearized
equations'' in many KAM problems. In such a case we have a very clear
and simple description of the smooth cohomology: it can be shown the
Haar measure on $\T^d$ is the only (modulo multiplication by a
constant) $R_\alpha$-invariant distribution, and $R_\alpha$ is
cohomologically $C^\infty$-stable if and only if $\alpha$ is a
Diophantine vector (see \S\ref{sec:arithmetics} for definitions).

Nevertheless, the general situation is much more complicated: when $f$
is an arbitrary quasi-periodic diffeomorphism, \ie
$f\in\Diff{}\infty(\T^d)$ is topologically conjugate to an ergodic
rigid rotation, in general it is very hard to determine the space of
$f$-invariant distributions and the cohomological stability issue
seems to be even subtler.

For instance, the problem of computing the $C^\infty$ first cohomology
space of an arbitrary minimal circle diffeomorphism has been included
in several compilations of open problems concerning group actions and
foliations (see \cite{langevin,ghys} for instance).

In this article we solve this problem proving the following
\begin{maintheorem}
  Let $F\colon\T\to\T$ be an orientation-preserving $C^\infty$
  diffeomorphism with irrational rotation number and $\mu$ be its only
  invariant probability measure. Then, up to multiplication by a real
  constant, $\mu$ is the only $F$-invariant distribution.
\end{maintheorem}

Yoccoz, in his PhD thesis, showed that within the set of smooth circle
diffeomorphisms with fixed rotation number
$\alpha\in(\R\setminus\Q)/\Z$ those which are smoothly conjugate to
the rotation $R_\alpha:x\mapsto x+\alpha$ form a dense subset (see
Chapter III in \cite{yoccoz-thesis}). To some extent, our Theorem A
can be considered as a ``cocycle version'' of his result.

It is important to remark that Theorem A is absolutely one-dimensional
and cannot be extended to higher dimensions. In fact, in a forthcoming
article~\cite{avila-kocsard} we will show the existence of smooth
diffeomorphisms of $\T^2$ which are topologically conjugate to rigid
rotations and exhibit higher order invariant distributions.

On the other hand, as an almost straightforward consequence of
Theorem~A, we can obtain the following
\begin{maincorollary-I}
  A minimal $C^\infty$ circle diffeomorphism is cohomologically
  $C^\infty$-stable if and only if its rotation number is Diophantine.
\end{maincorollary-I}

\subsection{Denjoy-Koksma inequality improved}
\label{sec:den-kok-impr}

Given a circle homeomorphism $F$ with irrational rotation number
$\rho(F)$ and a real function $\phi\colon\T\to\R$, the classical
Denjoy-Koksma inequality affirms that the Birkhoff sums satisfy
\begin{equation}
  \label{eq:denjoy-koksma-in}
  \left| \sum_{i=1}^{q_n-1}\phi(f^i(x)) - q_n\int_{\T}\phi\dd\mu
  \right| \leq \Var(\phi),\quad\forall x\in\T,
\end{equation}
whenever $\phi$ has bounded variation, $\mu$ is the only $F$-invariant
probability measure and $q_n$ is a denominator of a rational
approximation of $\rho(F)$ given by the continued fraction algorithm
(see \S~\ref{sec:cont-frac} and Proposition~\ref{pro:denjoy-koksma}
for details).

Nevertheless, when $F$ is a $C^3$ diffeomorphism and $\phi$ is the
\emph{log-derivative cocycle}, \ie $\phi=\log DF$, Herman showed (see
Corollary 2.5.2 of Chapter VII in \cite{herman-ihes}) that the
previous estimate can be improved. In fact, he proved that $\log
DF^{q_n}=\sum_{i=0}^{q_n-1}\log DF\circ F^i$ converges uniformly to
zero, as $n\to\infty$. The interested reader can also find a
\emph{hard} version of this result in \cite{yoccoz-thesis}.

Here, as a consequence of Theorem~\ref{thm:main-thm-finitary}, which
is nothing but a finite regularity version of Theorem~A, we get the
following result which can be considered as a generalization of Herman
result for arbitrary cocycles:
\begin{maincorollary-II}
  If $F$ is $C^{11}$ and $\phi$ is $C^1$, it holds
  \begin{displaymath}
    \left\|\sum_{i=0}^{q_n-1}\phi\circ F^i
      -q_n\int_\T\phi\dd\mu\right\|_{C^0} \to 0,\quad\text{as }
    n\to\infty. 
  \end{displaymath}
\end{maincorollary-II}

\subsection{Some open questions}
\label{sec:questions}

At this point it seems to be natural to analyze the first
cohomological space of higher dimension quasi-periodic
diffeomorphisms.

As we have already mentioned above, in a forthcoming
article~\cite{avila-kocsard} we will show there exist quasi-periodic
diffeomorphisms on higher dimensional tori exhibiting higher order
invariant distributions.

However, all the examples we know so far have Liouville rotation
vectors and are cohomologically $C^\infty$-unstable. So it is
reasonable to propose the following

\begin{question}
  Let $\alpha\in\R^d$ be an irrational vector and
  $f\in\Diff+\infty(\T^d)$ be topologically conjugate to the rigid
  rotation $R_\alpha : x \mapsto x+\alpha$.

  Is it true that $f$ is cohomologically $C^\infty$-stable if an only
  if $\alpha$ is Diophantine?
\end{question}

\begin{question}
  \label{ques:Diophantine-inv-dist-Td}
  Let $\alpha$ and $f$ as above. If $\alpha$ is Diophantine, then does
  it hold
  \begin{displaymath}
    \dim\Dis{}(f)=1?
  \end{displaymath}
\end{question}

It is interesting to remark that
Question~\ref{ques:Diophantine-inv-dist-Td} could be a first step
toward an eventually higher dimensional version of Herman-Yoccoz
linearization theorem.

\subsection*{Acknowledgments}

We are very grateful to Giovanni Forni for several useful discussions
and for having pointed out a (serious) mistake in the very first proof
of Theorem 6.1.

A.K. would also like to thank \'Etienne Ghys for having brought the
problem that motivated Theorem A to his attention, and Anatole Katok
for insightful conversations.

We thank the anonymous referee for his/her valuable suggestions that
helped to improve the exposition of the paper.

This research was partially conducted during the period A.A. served as
a Clay Research Fellow. A.K. was partially supported by FAPERJ-Brazil.

\section{Preliminaries}
\label{sec:prelim}

\subsection{General notations}
\label{sec:gen-notation}

Along this article $M$ will denote an arbitrary smooth boundaryless
manifold. Given $r\in\bb{N}_0\cup\{\infty\}$, we write $C^r(M)$ for
the space of $C^r$ real functions on $M$ and $\Diff{}r(M)$ for the
group of $C^r$ diffeomorphisms\footnote{As usual, we use the term
  ``$C^0$ diffeomorphism'' as a synonymous of ``homeomorphism''.}.

Let us recall that when $r$ is finite the \emph{uniform
  $C^r$-topology} turns $C^r(M)$ into a Banach space. On the other
hand, we shall consider the space $C^\infty(M)$ endowed with its usual
Fr\'echet topology, which can be defined as the projective limit of
the family of Banach spaces $(C^r(M))_{r\in\bb{N}}$.

Given any $f\in\Diff{}{r}(M)$, $\Fix(f)$ and $\Per(f)$ stand for the
set of fixed and periodic points of $f$, respectively. Whenever $M$ is
orientable, we write $\Diff{+}{r}(M)$ for the subgroup of $C^r$
orientation-preserving diffeomorphisms.

The $d$-dimensional torus will be denoted by $\T^d$ and will be
identified with $\R^d/\Z^d$. The canonical quotient projection will be
denoted by $\pi\colon\R^d\to\T^d$. For simplicity, we shall just write
$\T$ for the $1$-torus, \ie the circle.

The symbol $\Leb_d$ will be used to denote the Lebesgue measure on
$\R^d$, as well as the Haar probability measure on $\T^d$. Once again,
for the sake of simplicity, we just write $\Leb$, and also $\dd x$,
instead of $\Leb_1$.

As usual, we shall identify $C^r(M,\R^k)$ with $(C^r(M))^k$, and
$C^r(\T^d)$ with the space of $\Z^d$-periodic real $C^r$ functions on
$\R^d$.

In the particular case of $C^r$ real functions on $\T$, we explicitly
define the \emph{$C^r$-norm} on $C^r(\T)$ (with $0\leq r<\infty$) by
\begin{displaymath}
  \|\phi\|_{C^r}:=\max_{x\in\T}\max_{0\leq j\leq r} |D^j\phi(x)|,
  \quad\forall \phi\in C^r(\T).
\end{displaymath} 

Moreover, whenever $I\subset\R$ is a compact interval and $\psi\in
C^r(\R)$, we define
\begin{displaymath}
  \left\|\psi\big|_{I}\right\|_{C^r}:=\max_{x\in I}\max_{0\leq j\leq r}
  |D^j\psi(x)|.
\end{displaymath}

Next, we define the space of lifts of circle diffeomorphisms by
\begin{displaymath}
  \widetilde{\Diff+r}(\T):=\left\{f\in\Diff+r(\R):
    f-id_{\R}\in C^r(\T)\right\}.
\end{displaymath}
It can be easily shown that this space is connected and simply
connected. In particular, this space can be identified with the
universal covering of $\Diff+r(\T)$. Making some abuse of notation, we
will also denote by $\pi$ the canonical projection
$\widetilde{\Diff+r}(\T)\to\Diff+r(\T)$ that associates to each
$f\in\widetilde{\Diff+r}(\T)$ the only circle diffeomorphism lifted by
$f$.

As usual, we write $\rho\colon\widetilde{\Diff+0}(\T)\to\bb{R}$ for
the \emph{rotation number} function, and we will use the same letter
to call the induced map $\rho\colon\Diff+0(\T)\to\bb{R}/\bb{Z}$ (see
\S1.1 in \cite{welington} for the definitions).

And finally, two important remarks about notation. First, for the sake
of simplicity, when dealing with estimates we will use the letter $C$
to denote any positive real constant which may assume different values
along the article, even in a single chain of inequalities.  

And secondly, we will denote the intervals of the real line regardless
of the order of the extremal points, \ie if $a,b$ are two different
points of $\bb{R}$, we shall write $(a,b)$ for the only bounded
connected component of $\bb{R}\setminus\{a,b\}$, independently of the
order of the points. Of course, we will follow the same convention for
the intervals $[a,b]$, $[a,b)$ and $(a,b]$.

\subsection{Cocycles, coboundaries and invariant distributions}
\label{sec:cocy-coub-distrib}

From now on let us assume our manifold $M$ is closed, \ie compact and
boundaryless, and let $f\in\Diff{}r(M)$, with
$r\in\bb{N}_0\cup\{\infty\}$. Every $\psi\in C^k(M)$, where $0 \leq
k\leq r$, can be considered as a $C^k$ (real) \emph{cocycle} over $f$
writing
\begin{displaymath}
  M\times\bb{Z}\ni (x,n) \mapsto \Bs^n\psi(x)\in\bb{R},
\end{displaymath}
where $\Bs^n\psi=\Bs^n_f\psi$ denotes the \emph{Birkhoff sum over $f$}
given by
\begin{displaymath}
  \Bs^n\psi:=
  \begin{cases}
    \sum_{i=0}^{n-1}\psi\circ f^i & \text{if } n\geq 1; \\
    0 & \text{if } n=0; \\
    -\sum_{i=1}^{-n}\psi\circ f^{-i} & \text{if } n<0.
  \end{cases}
\end{displaymath}

We say that the cocycle $\psi\in C^k(M)$ is a $C^\ell$
\emph{coboundary}, with $0\leq \ell\leq k\leq r$, whenever there
exists $u\in C^\ell(M)$ solving the following cohomological equation:
\begin{displaymath}
  u\circ f -u =\psi.
\end{displaymath}
We say that $\phi,\psi\in C^k(\T)$ are $C^\ell$-\emph{cohomologous}
whenever the function $\phi-\psi$ is a $C^\ell$ coboundary.

The space of $C^\ell$ coboundaries will be denoted by
$B(f,C^\ell(M))$, and since it is clearly a linear subspace of
$C^\ell(M)$, we can define
\begin{displaymath}
  H^1(f,C^\ell(M)):=C^\ell(M)/B(f,C^\ell(M)), 
\end{displaymath}
called the \emph{first $C^\ell$-cohomology space of $f$}.

This space $H^1(f,C^\ell(M))$ naturally inherits the quotient topology
from $C^\ell(M)$. Unfortunately, in general $B(f,C^\ell(M))$ is not
closed in $C^\ell(M)$, and therefore, this quotient topology is
non-Hausdorff. So, it is reasonable to propose the following

\begin{definition}
  \label{def:cohomological-stability}
  We say that $f$ is \emph{cohomologically $C^\ell$-stable} whenever
  $B(f,C^\ell(M))$ is a closed in $C^\ell(M)$. On the other hand, we
  define the \emph{first reduced $C^\ell$-cohomology space} as being
  \begin{displaymath}
    \tilde{H}^1(f,C^\ell(M)):=C^\ell(M)/\cl_\ell(B(f,C^\ell(M))),
  \end{displaymath}
  where $\cl_\ell(\cdot)$ denotes the closure in the uniform
  $C^\ell$-topology.
\end{definition}

As we have already mentioned in \S\ref{sec:intro}, the study of the
structure of the spaces $H^1(f,C^\infty(M))$ and
$\tilde{H}^1(f,C^\infty(M))$ naturally leads us to consider the space
of $f$-invariant \emph{(Schwartz) distributions} on $M$.

So, for each $k\in\bb{N}_0$, let $\Dis{k}(M)$ be the topological dual
space of $C^k(M)$, \ie the space of \emph{distributions of order up to
  $k$} of $M$. As usual, the dual space of $C^\infty(M)$ will be
simply denoted by $\Dis{}(M)$.

Since all the inclusions $C^{k+1}(M)\hookrightarrow C^k(M)$ and
$C^\infty(M)\hookrightarrow C^k(M)$ are continuous and have dense
range, we can suppose we have the following chain of inclusions, which
are defined modulo unique extensions:
\begin{displaymath}
  \Dis0(M)\subset\Dis1(M)\subset\Dis2(M)\subset\ldots\subset
  \Dis{}(M).
\end{displaymath}
Moreover, since we are assuming $M$ is compact, it is well-known that
\begin{displaymath}
  \Dis{}(M)=\bigcup_{k\geq 0}\Dis{k}(M).
\end{displaymath}

On the other hand, any $C^k$ diffeomorphism $f$ acts linearly on
$C^k(M)$ by pull-back, and the adjoint of this action is the linear
operator $f_*\colon\Dis{k}(M)\to\Dis{k}(M)$ given by
\begin{displaymath}
  \langle f_*T,\psi\rangle:=\langle T,\psi\circ f\rangle,\quad
  \forall\:T\in\Dis{k}(M),\ \forall\:\psi\in C^k(M).
\end{displaymath}

In this case, $\Fix(f_*)$ is the space of \emph{$f$-invariant
  distributions of order up to $k$} and it will be denoted by
$\Dis{k}(f)$. Of course, it holds $\Dis{}(f)=\bigcup_{k\geq
  0}\Dis{k}(f)$.

As we mentioned above, there is a tight relation between the space of
invariant distributions $\Dis{}(f)$ and the reduced cohomology group
$\tilde{H}^1(f,C^\infty(M))$. In fact, as a straightforward
consequence of Hahn-Banach theorem we get the following

\begin{proposition}
  \label{pro:inv-dist-vs-cobound}
  Given any $f\in\Diff{}{k}(M)$, with $k\in\bb{N}_0\cup\{\infty\}$, it
  holds
  \begin{displaymath}
    \cl_{k}(B(f,C^k(M)))=\bigcap_{T\in\Dis{k}(f)}\ker T.
  \end{displaymath}
  In particular, this implies that
  \begin{displaymath}
    \dim \tilde{H}^1(f,C^k(M))=\dim\Dis{k}(f).
  \end{displaymath}
\end{proposition}

\subsection{Arithmetic}
\label{sec:arithmetics}

\subsubsection{Diophantine and Liouville vectors}
\label{sec:dioph-liouville}

For any $\boldsymbol{\alpha}\in\bb{R}^d$ we write
\begin{displaymath}
  \|\boldsymbol{\alpha}\|:=\mathrm{dist}
  (\boldsymbol{\alpha},\bb{Z}^d).
\end{displaymath}
Notice that, since $\|\boldsymbol{\alpha}+\mathbf{n}\|=
\|\boldsymbol{\alpha}\|$ for every $\mathbf{n}\in\bb{Z}^d$, we can
naturally consider $\|\cdot\|$ as defined on $\bb{T}^d$.

We say $\boldsymbol{\alpha}=(\alpha_1,\ldots,\alpha_d)\in\bb{R}^d$ is
\emph{irrational} if and only if, for any
$(n_1,\ldots,n_d)\in\bb{Z}^{d}$, it holds
\begin{equation}
  \label{eq:dioph-vector-def}
  \bigg\|\sum_{i=1}^{d} n_i\alpha_i\bigg\|=0 \implies n_i=0,\
  \text{for } i=1,\ldots, d.
\end{equation}

Moreover, the vector $\boldsymbol{\alpha}$ is said to be
\emph{Diophantine} whenever there exist constants $C,\tau>0$
satisfying
\begin{displaymath}
  \bigg\|\sum_{i=1}^d\alpha_iq_i\bigg\|\geq\frac{C}{\max_i|q_i|^\tau}, 
\end{displaymath}
for every $(q_1,\ldots,q_d)\in\bb{Z}^d\setminus\{0\}$. On the other
hand, an irrational element of $\bb{R}^d$ which is not Diophantine is
called \emph{Liouville}.

\subsubsection{Continued fractions}
\label{sec:cont-frac}

In this paragraph we introduce some common notations and recall some
elementary and well-known results about continued fractions
(see~\cite{HardyWright} for details).

First of all, the \emph{Gauss map} $A\colon (0,1)\to [0,1)$ is defined
by
\begin{displaymath}
  A(x):=\frac{1}{x} - \left\lfloor\frac{1}{x}\right\rfloor.
\end{displaymath}

For each $\alpha\in\bb{R}\setminus\bb{Q}$ we can associate the
sequences $(\alpha_n)_{n\geq 0}$ and $(a_n)_{n\geq 0}$ which are
recursively defined by
\begin{gather}
  \label{eq:an-alphan-def}
  \alpha_0 := \alpha - \lfloor\alpha\rfloor, \quad \alpha_n :=
  A^n(\alpha_0),\quad\forall\:n\geq 1; \\
  a_0 := \lfloor \alpha\rfloor, \quad a_{n+1}:=\left\lfloor
    \frac{1}{\alpha_{n}}\right\rfloor,\quad\forall\:n\geq 0.
\end{gather}

The \emph{$n^{th}$-convergent} of $\alpha$ is defined by
\begin{displaymath}
  p_n/q_n:=a_0 + \cfrac{1}{a_1+\cfrac{1}{a_2+\cfrac{1}{\dotsb+
        \cfrac{1}{a_n}}}} 
\end{displaymath}
and the sequences $(p_n)_{n\geq -2}$ and $(q_n)_{n\geq -2}$ satisfy
the following recurrences:
\begin{gather}
  \label{eq:pn-def}
  p_{-2}:=0,\quad p_{-1}:=1,\quad p_n:=a_np_{n-1}+p_{n-2},
  \quad\forall\:n\geq 0; \\
  \label{eq:qn-def}
  q_{-2}:=1,\quad q_{-1}:=0,\quad q_n:=a_nq_{n-1}+q_{n-2},
  \quad\forall\:n\geq 0.
\end{gather}

A very important property about $(q_n)$ that we will repeatedly use in
the future is the following:
\begin{equation}
  \label{eq:q_n-as-closest-return}
  q_{n+1}=\min\{q\in\bb{N} : \|q\alpha\|<\|q_n\alpha\|\}, \quad\forall
  n\geq 1.
\end{equation}

The reader can easily show that the sequences $(p_n)$ and $(q_n)$
satisfy the following relation:
\begin{equation}
  \label{eq:pn-qn-SL}
  p_{n-1}q_n-p_nq_{n-1}=(-1)^{n},\quad\forall\:n\geq -1. 
\end{equation}

Now let us define the sequence $(\beta_n)_{n\geq -1}$ by
\begin{align}
  \label{eq:betan-def}
  \beta_{-1} &:=1, \\
  \beta_n & :=\prod_{i=0}^n\alpha_i, \quad \forall\:n\geq 0.
\end{align}

By straightforward computations we can show that the sequence
$(\beta_n)$ satisfies
\begin{equation}
  \label{eq:beta_n-q_n-alpha-p_n}
  \beta_n=(-1)^n(q_n\alpha-p_n)> 0,
\end{equation}
and
\begin{equation}
  \label{eq:beta_n-estimate}
  \frac{1}{q_n+q_{n+1}}< \beta_n < \frac{1}{q_{n+1}},
\end{equation}
for every $n\geq 0$.  

On the other hand, the growth of the sequences $(q_n)$ and $(\beta_n)$
determines whether the number $\alpha$ is Diophantine or Liouville: if
$\tau$ denotes any positive real number and we write
\begin{equation}
  \label{eq:L-alpa-delta}
  \Lie(\alpha,\tau):=\left\{ m\in\bb{N} :
    \beta_m<\beta_{m-1}^\tau \right\},
\end{equation}
it is very easy to verify that $\alpha$ is Liouville if and only if
$\Lie(\alpha,\tau)$ has infinitely many elements, for every
$\tau>1$. In fact, this can be proved rewriting
(\ref{eq:dioph-vector-def}) for $d=1$: we have $\alpha$ is Diophantine
if and only if there exist constant $C,\tau>0$ such that
\begin{equation}
  \label{eq:dioph-number-q_n}
  \beta_n=|q_n\alpha-p_n| > \frac{C}{q_n^{1+\tau}}, \quad\forall n\geq
  0,
\end{equation}
and, by estimate (\ref{eq:beta_n-estimate}), this is equivalent to
\begin{equation}
  \label{eq:dioph-number-beta_n}
  \beta_{n+1}> C\beta_n^{1+\tau}, \quad\forall n\geq 0. 
\end{equation}

\subsection{Cohomology of minimal rotations on the torus}
\label{sec:cohomo-rigid-rot}

Let $\boldsymbol{\alpha}=(\alpha_1,\ldots,\alpha_d)$ be an irrational
vector in $\bb{R}^d$. It is well-known that in such a case the
rotation $R\colon\bb{T}^d\to\bb{T}^d$ given by
\begin{displaymath}
  R:x\mapsto x + (\boldsymbol{\alpha}+\bb{Z}^d),
\end{displaymath}
is minimal (\ie all its orbits are dense in $\T^d$) and uniquely
ergodic, being the Haar measure $\Leb_d$ the only $R$-invariant Borel
probability measure.

Moreover, we have the following result which belongs to the
\emph{folklore}:
\begin{proposition}
  \label{pro:one-dim-inv-dist}
  Continuing with the notation we introduced above, we have
  \begin{displaymath}
    \Dis{}(R)=\bb{R}\Leb_d.
  \end{displaymath}
  
  On the other hand, $R$ is cohomologically $C^\infty$-stable if and
  only if $\boldsymbol{\alpha}$ is Diophantine.
\end{proposition}

\begin{proof}
  Let $\psi\in C^\infty(\bb{T}^d)$ be such that
  $\int\psi\:\mathrm{d}\Leb_d=0$, and let us consider the Fourier
  development of $\psi$:
  \begin{displaymath}
    \psi(x)=\sum_{\mathbf{k}\in\bb{Z}^d\setminus\{0\}}\hat{\psi}_k
    e^{2\pi i\mathbf{k}\cdot x}.
  \end{displaymath}

  Then, the Fourier coefficients of any (integrable) solution of the
  cohomological equation $\psi=uR-u$ must satisfy the following
  relation:
  \begin{equation}
    \label{eq:u-Fourier-coef-def}
    \hat{u}_\mathbf{k}:=\frac{\hat{\psi}_\mathbf{k}}{e^{2\pi
        i\mathbf{k}\cdot\boldsymbol{\alpha}}-1},\quad\forall\:
    \mathbf{k}\in\bb{Z}^d\setminus\{0\}.  
  \end{equation}

  If $\{U_n\}_{n\geq 1}$ is any sequence of finite subsets of
  $\bb{Z}^d$ such that $\bigcup_{n\geq 1}U_n=\bb{Z}^d\setminus\{0\}$
  and $U_{n}\subset U_{n+1}$, for every $n\geq 1$, then one can define
  the trigonometric polynomials
  \begin{align*}
    \psi_n(x) & :=\sum_{\mathbf{k}\in U_n}
    \hat{\psi}_\mathbf{k}e^{2\pi i\mathbf{k}\cdot x}; \\
    u_n(x) & := \sum_{\mathbf{k}\in U_n} \hat{u}_\mathbf{k}e^{2\pi
      i\mathbf{k}\cdot x}, \quad\text{for } n\geq 1.
  \end{align*}

  Since $u_nR - u_n =\psi_n$, we have $\psi_n\in B(R,C^\infty(\T^d))$,
  and clearly $\psi_n\to\psi$ in the $C^\infty$-topology. Thus,
  $\psi\in\cl_\infty(B(f,C^\infty(\T^d)))$. By
  Proposition~\ref{pro:inv-dist-vs-cobound}, we conclude that
  $\Dis{}(R)=\bb{R}\Leb_d$.

  Now, when $\boldsymbol{\alpha}$ is Diophantine it is easy to verify
  that the Fourier coefficients $(\hat
  u_{\mathbf{k}})_{\mathbf{k}\in\Z^d}$ decay sufficiently fast at
  infinity to guarantee that
  \begin{displaymath}
    u(x):=\sum_{\mathbf{k}\in\bb{Z}^d\setminus\{0\}}
    \hat{u}_\mathbf{k}e^{2\pi i\mathbf{k}\cdot x},
  \end{displaymath}
  defines a $C^\infty$ function which turns to be a solution for the
  cohomological equation $uR - u=\psi$. Applying
  Proposition~\ref{pro:inv-dist-vs-cobound} once again, we obtain
  \begin{displaymath}
    B(f,C^\infty(\T^d))=\ker\Leb_d=\cl_\infty(B(f,C^\infty(\T^d))), 
  \end{displaymath}
  and therefore, $R$ is cohomologically $C^\infty$-stable.

  On the other hand, when $\boldsymbol{\alpha}$ is Liouville it is
  possible to find a sequence $(\mathbf{n}_j)_{j\geq
    1}\subset\Z^d\setminus\{0\}$ satisfying $\mathbf{n}_j\to
  \infty$, as $j\to+\infty$ and
  \begin{displaymath}
    \left\|\sum_{i=1}^d\alpha_in_{j,i}\right\|\leq
    \frac{1}{\max_i|n_{j,i}|^j}, \quad\forall\: j\geq 1,
  \end{displaymath}
  where, of course, $\mathbf{n}_j=(n_{j,1},\ldots,n_{j,d})$.

  This clearly implies that writing 
  \begin{displaymath}
    \psi(x):=\sum_{j\in\bb{N}} \Big[\big(e^{2\pi
      i\mathbf{n}_j\cdot\alpha}-1\big) e^{2\pi i
      \mathbf{n}_j\cdot x} + \big(e^{-2\pi
      i\mathbf{n}_j\cdot\alpha}-1\big) e^{-2\pi i
      \mathbf{n}_j\cdot x}\Big]
  \end{displaymath}
  we get $\psi\in C^\infty(\T^d)\cap\ker\Leb_d$. However, $\psi\not\in
  B(R,C^\infty(\T^d))$ because the Fourier coefficients $(\hat
  u_{\mathbf{k}})$ of an eventual solution of the cohomological
  equation given by (\ref{eq:u-Fourier-coef-def}) satisfy
  \begin{displaymath}
    \hat u_{\mathbf{n}_j}=1, \quad\forall\: j\in\bb{N},
  \end{displaymath} 
  with $\mathbf{n}_j\to\infty$, as $j\to\infty$.
\end{proof}

\section{Proof of the corollaries }
\label{sec:corollaries}

Since the proof of Theorem A is rather technical, we will start
proving corollaries B and C assuming Theorem A. 

\subsection{Corollary B}
\label{sec:corollary-B}

Let $F\in\Diff+\infty(\T)$ be such that
$\rho(F)\in(\R\setminus\Q)/\Z$, $\mu$ denote the only $F$-invariant
probability measure, and $f\in\widetilde{\Diff+\infty}(\T)$ be a lift
of $F$.

By the unique ergodicity it easily follows that
\begin{equation}
  \label{eq:int-log-Df-mu-0}
  \int_{\T}\log Df\dd\mu=0.
\end{equation} 

On the other hand, by Theorem~A we know $\Dis{}(F)=\R\mu$. Hence, if
we suppose that $F$ is cohomologically $C^\infty$-stable, by
(\ref{eq:int-log-Df-mu-0}) and
Proposition~\ref{pro:inv-dist-vs-cobound} we conclude $\log Df\in
B(F,C^\infty(\T))$, \ie there exists $u\in C^\infty(\T)$ satisfying
\begin{equation}
  \label{eq:logDf-cohomo-eq}
  uF-u=\log Df. 
\end{equation}

Now, let us write
\begin{displaymath}
  h^\prime:=C^{-1}\exp(-u)\in C^\infty(\T),
\end{displaymath}
where $C:=\int_\T\exp(-u)\dd\Leb$. So, in particular $h^\prime$ is
positive and satisfies $\int_\T h^\prime\dd\Leb=1$. Hence, defining
\begin{displaymath}
  h(x):=\int_0^x h^\prime(t)\dd t, \quad\forall x\in\R, 
\end{displaymath}
we have $h\in\widetilde{\Diff+\infty}(\T)$, and by
(\ref{eq:logDf-cohomo-eq}) we get
\begin{displaymath}
  D(h\circ f)=(Dh\circ f)Df = (h^\prime\circ f)Df= h^\prime=Dh.
\end{displaymath}
This implies there exists $\rho\in\R$ such that $hf=h+\rho$, and by
invariance of the rotation number under conjugacy we have
$\rho=\rho(f)$.

Therefore, $F$ is $C^\infty$-conjugate to the rigid rotation $x\mapsto
x+\rho(F)$, and applying Proposition~\ref{pro:one-dim-inv-dist} we
conclude $\rho(F)$ is Diophantine.

Reciprocally, by Herman-Yoccoz theorem~\cite{herman-ihes,yoccoz-ens}
any circle diffeomorphism with Diophantine rotation number is
$C^\infty$-conjugate to the rigid rotation, and by
Proposition~\ref{pro:one-dim-inv-dist} it must be cohomologically
$C^\infty$-stable, as desired.

\subsection{Corollary C}
\label{sec:corollary-c}

To prove Corollary~C we will use two different arguments depending on
how well we can approximate the rotation number of the diffeomorphism
by rational numbers.

When the rotation number is badly approximated, we will use the finite
regularity version of Yoccoz linearization theorem (Th\'eor\`eme, page
335 in \cite{yoccoz-ens}). On the other hand, when the rotation number
is not ``too badly'' approximated, we will use a finite regularity
version of our Theorem~A, that is Theorem~\ref{thm:main-thm-finitary}.

So, let $F\in\Diff+{11}(\T)$ be a minimal diffeomorphism,
$f\in\widetilde{\Diff+{11}}(\T)$ be a lift of $F$, $\mu$ be the only
$F$-invariant probability measure, and $\phi\in C^1(\T)$. Let $(p_n)$
and $(q_n)$ be the sequences associated to $\alpha:=\rho(f)$ given by
(\ref{eq:pn-def}) and (\ref{eq:qn-def}), respectively, and
$\varepsilon>0$ be arbitrary. Then, let us consider the set
$\Lie(\alpha,11/2)$ given by (\ref{eq:L-alpa-delta}).

First, let us suppose $\Lie(\alpha,11/2)$ is finite. Thus, as we have
already mentioned at the end of \S\ref{sec:cont-frac}, in this case
$\alpha$ is Diophantine. More precisely, there exists $C>0$ such that
\begin{equation}
  \label{eq:alpha-dioph-11-2}
  \left\|q\alpha\right\| \geq \frac{C}{q^{11/2}}, \quad\forall
    q\in\bb{N}.
\end{equation}

Then, by Yoccoz linearization theorem~\cite{yoccoz-ens} we have $F$ is
$C^1$-conjugate to the rigid rotation $R_\alpha$. Therefore, applying
Proposition~\ref{pro:one-dim-inv-dist} we can conclude
$\Dis{1}(F)=\R\mu$. 

On the other hand, if we suppose $\Lie(\alpha,11/2)$ has infinite
elements, we can apply Theorem~\ref{thm:main-thm-finitary} to conclude
that $\Dis{1}(F)=\R\mu$, too. 

In any case, applying Proposition~\ref{pro:inv-dist-vs-cobound} we can
conclude there exists $u\in C^1(\T)$ such that
\begin{displaymath}
  \left\|(uF-u)-\Big(\phi-\int_\T\phi\dd\mu\Big)\right\|_{C^1}
  \leq\frac{\varepsilon}{2}.
\end{displaymath}
Writing 
\begin{displaymath}
  \tilde\phi:=uF-u + \int_\T\phi\dd\mu,
\end{displaymath}
it clearly holds
\begin{displaymath}
  \int_\T\tilde\phi\dd\mu=\int_\T\phi\dd\mu.
\end{displaymath}
Hence, $\|\tilde\phi-\phi\|_{C^1}\leq\varepsilon/2$ and
\begin{displaymath} 
  \left(\tilde\phi-\int_\T\tilde\phi\dd\mu\right)\in B(F,C^1(\T)). 
\end{displaymath}

On the other hand, if we write
\begin{displaymath}
  M_n:=\sup_{x\in\R} \left|f^{q_n}(x)-x-p_n\right|, 
\end{displaymath}
by the minimality of $F$ we get $M_n\to 0$ when $n\to\infty$, and so,
there exists $N\in\bb{N}$ such that
$\|Du\|_{C^0}M_n\leq\varepsilon/2$, provided $n\geq N$.

Finally, applying Denjoy-Koksma inequality (\ref{eq:denjoy-koksma-in})
(see Proposition~\ref{pro:denjoy-koksma} for the precise statement)
for $\tilde\phi-\phi$ we get
\begin{displaymath}
  \begin{split}
    \bigg|\Bs^{q_n}\phi(x) - &q_n\int_\T\phi\dd\mu\bigg| \\
    & \leq \left|\Bs^{q_n}(\phi-\tilde\phi)(x) -
      q_n\int_\T(\phi-\tilde\phi)\dd\mu \right| +
    \left|\Bs^{q_n}\tilde\phi(x)-q_n\int_\T\tilde\phi\dd\mu\right| \\
    & \leq \Var(\phi-\tilde\phi) + \left|u(F^{q_n}(x))-u(x)\right| \\
    & \leq \int_\T|D(\phi-\tilde\phi)|\dd\mu + \|Du\|_{C^0} M_n \leq
    \|\phi-\tilde\phi\|_{C^1} + \frac{\varepsilon}{2} \leq
    \varepsilon,
  \end{split}
\end{displaymath}
for every $x\in\T$ and provided $n$ is sufficiently big. Since
$\varepsilon$ is arbitrary, Corollary C is proved.

\section{$C^r$-estimates for real cocycles}
\label{sec:estimate-cocycles}

This section can be considered the starting point of the proof of
Theorem A. The principal new result here is
Proposition~\ref{pro:Cr-estimate}, which is mainly inspired in the
work of Yoccoz~\cite{yoccoz-ens}.

Along this section $F$ will denote an arbitrary orientation-preserving
diffeomorphism of $\T$ with irrational rotation number. Once and for
all we fix a lift $f\colon\bb{R}\to\bb{R}$ of $F$, and to simplify the
exposition, we write $\alpha:=\rho(f)\in\bb{R}\setminus\bb{Q}$.

Using the notation we introduced in \S\ref{sec:arithmetics}, let
$(a_n)$, $(\alpha_n)$, $(\beta_n)$, $(p_n)$, $(q_n)$ be the sequences
associated to $\alpha$ defined by (\ref{eq:an-alphan-def}),
(\ref{eq:pn-def}), (\ref{eq:qn-def}) and (\ref{eq:betan-def}).

For each $n\geq 0$ and $\phi\colon \T\to\R$, we define $f_n$ and
$\phi_n$ by
\begin{equation}
  \label{eq:f_n-def}
  f_n:=f^{q_n}-p_n,
\end{equation}
and
\begin{equation}
  \label{eq:phi_n-def}
  \phi_n:=\Bs^{q_n}\phi=\sum_{i=0}^{q_n-1}\phi\circ f^i. 
\end{equation}

And for any $x\in\bb{R}$ and $n\geq 0$, we consider the following
closed intervals:
\begin{equation}
  \label{eq:In-Jn-Kn-defs}
  \begin{split}
    I_n(x)&:=[x,f_n(x)], \\
    J_n(x)&:=[f_{n+1}(x),f_n(x)], \\
    K_n(x)&:=[f_n^{-2}(x),f_n(x)].
  \end{split}
\end{equation}
Let us recall that according to our notation conventions (see
\S\ref{sec:prelim}), we denote intervals of the real line not taking
into account the order of their extreme points.

On the other hand, for any $\hat x\in\T$, we will write $\hat I_n(\hat
x)$, $\hat J_n(\hat x)$, and $\hat K_n(\hat x)$ to denote the
intervals $\pi(I_n(x)),\:\pi(J_n(x)),\:\pi(K_n(x))\subset\T$,
respectively, where $x$ is any point in $\pi^{-1}(\hat x)\subset\R$.

The reader can find the proof of the following simple and classical
result in the book \cite{welington} (Lemma 1.3 of Ch. I):

\begin{lemma}
  \label{lem:order-orbits}
  Given any $n\geq 0$ and any $x\in\T$, the interior of the intervals
  $\hat I_n(x)$, $\hat I_n(F(x))$, \ldots, $\hat
  I_n(F^{q_{n+1}-1}(x))$ are two-by-two disjoint. In particular, it
  holds
  \begin{displaymath}
    \begin{split}
      J_n(x)&=I_{n+1}(x)\cup I_n(x), \\
      K_n(x)&=I_n(f_n^{-2}(x))\cup I_n(f_n^{-1}(x))\cup I_n(x).
    \end{split}
  \end{displaymath}
\end{lemma}

We will need the following notation:
\begin{equation}
  \label{eq:m-mn-Mn-def}
  \begin{split}
    m_n(x) &= |f_n(x)-x|=\Leb(I_n(x)), \\
    M_n &:= \max_{x\in\bb{R}} m_n(x).
  \end{split}
\end{equation}
 
For the sake of completeness, let us recall the Denjoy-Koksma
inequality, (see \cite{herman-ihes} for instance):

\begin{proposition}[Denjoy-Koksma Inequality]
  \label{pro:denjoy-koksma}
  If $F$ is $C^0$ and $\phi\colon\bb{T}\to\bb{R}$ has bounded
  variation then, for each $n\geq 1$, it holds
  \begin{equation}
    \label{eq:denjoy-koksma}
    \left\|\phi_n-q_n\int_{\bb{T}}\phi\:\mathrm{d}\mu
    \right\|_{C^0}\leq \Var(\phi),
  \end{equation}
  where $\mu$ denotes the unique $F$-invariant probability measure and
  $\Var(\phi)$ the total variation of $\phi$ over $\T$.
  
  On the other hand, for every $x\in\bb{R}$ and $1\leq k \leq q_{n+1}$
  it holds
  \begin{equation}
    \label{eq:bounded-dist}
    \left|\phi_k(y)-\phi_k(z)\right|\leq \Var(\phi),
    \quad\forall y,z\in I_n(x). 
  \end{equation}
\end{proposition}

\subsection{$C^r$-estimates for the log-derivative cocycle}
\label{sec:espacial-cocycle-log-Df}

When $F$ is $C^1$, we can consider a very particular cocycle which
plays a fundamental role in the analysis of the dynamical properties
of $F$: this is $\log DF=\log Df\in C^0(\T)$ and will be called the
\emph{log-derivative cocycle}.

The fundamental property of the log-derivative cocycle that turns it
to be so important is the chain rule:
\begin{displaymath}
  \Bs^k (\log Df) = \log Df^k,\quad \forall
  k\geq 1.   
\end{displaymath}

Applying Proposition~\ref{pro:denjoy-koksma} for the log-derivative
cocycle we can easily show:
\begin{corollary}
  \label{cor:m_n-variation}
  If $f$ is $C^2$, for every $n\geq 2$ we have
  \begin{displaymath}
    \max\left\{\left\|\log Df_n \right\|_{C^0}, \left\|\log Df_n^{-1}
      \right\|_{C^0}\right\} \leq \Var(\log Df). 
  \end{displaymath}

  And on the other hand, for every $x\in\R$ it holds
  \begin{displaymath}
    C^{-1} < \frac{m_n(x)}{m_n(y)} < C, \quad\forall y\in K_{n-1}(x), 
  \end{displaymath}
  where $C:=\exp(3\Var(\log Df))$.
\end{corollary}

The following tree propositions are due to Yoccoz~\cite{yoccoz-ens}:
\begin{proposition}
  \label{pro:sum-power-deriv}
  Let $f$ be $C^2$. Then, for every $\ell\geq 1$, $n\geq 2$ and $1\leq
  k \leq q_n$ it holds
  \begin{displaymath}
    \sum_{i=0}^{k-1} (Df^i(x))^\ell \leq C
    \frac{M_{n-1}^{\ell-1}}{m_{n-1}(x)^\ell}, \quad\forall x\in\T,     
  \end{displaymath}
  where $C=C(f):=\exp \left(\Var(\log Df)\right)$.
\end{proposition}

\begin{proposition}
  \label{pro:Ds-log-Dfn}
  Let $f$ be $C^r$, with $r\geq 3$, and $s$ be a natural number
  satisfying $1\leq s\leq r-1$. Then, there exists a real constant
  $C>0$, depending only on $f$ and $s$, such that for every $n\geq 2$
  and $1\leq k\leq q_n$ it holds
  \begin{displaymath}
    |D^s(\Bs^k\log Df) (x) | = |D^s\log Df^k(x)| \leq
    C\left(\frac{\sqrt{M_{n-1}}}{m_{n-1}(x)}\right)^s,
    \quad\forall x\in\T.
  \end{displaymath}
\end{proposition}

\begin{proposition}
  \label{pro:m_n-m_n-1-comparisson}
  Supposing $f$ is $C^3$, for every $n\geq 2$ and any $x\in\R$ there
  exist $y\in I_{n-1}(x)$ and $z\in I_n(x)$ such that
  \begin{displaymath}
    m_n(y)=\frac{\beta_n}{\beta_{n-1}}m_{n-1}(z)=\alpha_n m_{n-1}(z).   
  \end{displaymath}
\end{proposition}

We will also need the following estimate which, at some extend, can be
considered as an improvement of Proposition~\ref{pro:Ds-log-Dfn} for
$k=q_n$:
\begin{proposition}
  \label{pro:Ds-log-Dfn+1}
  Let us assume $f$ is $C^r$, with $r\geq 3$. Then there exists a
  constant $C>0$ and a natural number $n_0$ such that for any
  $x\in\R$, every $n\geq n_0$ and $0\leq s\leq r-2$ it holds
  \begin{displaymath}
    \left|D^s\log Df_n(y)\right| \leq C
    \frac{\sqrt{M_{n-1}}}{(m_{n-1}(x))^s} 
    \left[\big(\sqrt{M_{n-1}}\big)^{r-2} + \frac{m_n(x)}{m_{n-1}(x)}
    \right],
  \end{displaymath}
  for every $y\in K_{n-1}(x)$.
\end{proposition}

\begin{proof}
  See \S3.6 of Chapter III in \cite{yoccoz-thesis}.
\end{proof}

The following elementary formula relates the derivatives of the
Birkhoff sums of the log-derivative cocycles with the derivatives of
the iterates of the diffeomorphism. It can be found in
\cite{yoccoz-ens}, too:

\begin{proposition}
  \label{pro:Dg-from-DlogDg-formula}
  Given any $r\in\bb{N}_0$ and $g\in \Diff{+}{r+1}(\bb{R})$, we have
  \begin{displaymath}
    D^{r+1}g = P_r(D\log Dg, D^2\log Dg, \ldots, D^r\log Dg) Dg,
  \end{displaymath}
  where $P_r$ is the polynomial in $\bb{Z}[X_1,\ldots,X_r]$ defined
  inductively by $P_0=1$ and
  \begin{displaymath}
    P_{r+1}(X_1,\ldots,X_{r+1}):= X_1P_r(X_1,\ldots,X_r) +
    \sum_{i=1}^{r} X_{i+1} \frac{\partial P_r}{\partial
      X_i}(X_1,\ldots,X_r),  
  \end{displaymath}
  for every $r\geq 0$.

  In particular, all the polynomials $P_{r}$ satisfy
  \begin{equation}
    \label{eq:pseudo-homogen-polynoms-Pr}
    P_r(tX_1,t^2X_2,t^3X_3,\ldots,t^rX_r)=t^rP_r(X_1,\ldots,X_r). 
  \end{equation}
\end{proposition}

As a straightforward consequence of propositions~\ref{pro:Ds-log-Dfn}
and \ref{pro:Dg-from-DlogDg-formula}, we get
\begin{corollary}
  \label{cor:Dsg-from-DslogDg-estimate}
  Given $f$, $s$, $n$ and $k$ as in Proposition~\ref{pro:Ds-log-Dfn},
  there exists a constant $C>0$, depending only on $f$ and $s$, such
  that
  \begin{displaymath}
    |D^sf^k(x)|\leq C \left(\frac{\sqrt{M_n}}{m_n(x)}\right)^{s-1}
    Df^k(x), \quad\forall x\in\T.
  \end{displaymath}
\end{corollary}

\subsection{$C^r$-estimates for arbitrary real cocycles}
\label{sec:Cr-estimates-arbit-cocycles}

In this paragraph we shall concern with arbitrary real cocycles and
get some estimates for the (higher order) derivatives of them. The
main difference with the results of Yoccoz we recalled in
\S\ref{sec:Cr-estimates-arbit-cocycles} is that his estimates hold in
the whole circle and ours are rather \emph{``local''}.

Now we can get our first $C^1$ estimate for real cocycles:
\begin{proposition}
  \label{pro:C1-estimate}
  Let $f$ be $C^2$. Then, there exists a constant $C>0$ such that for
  every $\phi\in C^1(\T)$, any $n\geq 0$, any $x^\star\in\bb{R}$
  satisfying $m_n(x^\star)=M_n$, and every $y\in K_n(x^\star)$, it holds
  \begin{displaymath}
    \left|D(\Bs^k\phi)(y)\right|\leq C\frac{\|D\phi\|_{C^0}}{M_n},
    \quad\text{for } k=1,\ldots,q_{n+1}.  
  \end{displaymath}
\end{proposition}

\begin{proof}
  Applying estimate~(\ref{eq:bounded-dist}) to the log-derivative
  cocycle we get
  \begin{displaymath}
    |\log Df^i(y)- \log Df^i(z)|\leq  \Var(\log Df),  
  \end{displaymath}
  for every $x\in\bb{R}$, every $y,z\in I_n(x)$ and $0\leq i\leq
  q_{n+1}$. This clearly implies,
  \begin{equation}
    \label{eq:Df-boundedness}
    C^{-1} < \frac{Df^i(y)}{Df^i(z)} < C, \quad \forall y,z\in K_n(x),
  \end{equation}
  where $C:=\exp\left(3\|D\log Df\|_{L^1(\T)}\right)$.

  Then, combining Proposition~\ref{pro:sum-power-deriv} and estimate
  (\ref{eq:Df-boundedness}) we get, for every $y\in K_n(x^\star)$ and
  $1\leq k\leq q_{n+1}$,
  \begin{displaymath}
    \begin{split}
      \left|D(\Bs^k\phi)(y)\right| & =
      \left|\sum_{i=0}^{k-1}D\phi(f^i(y))Df^i(y)\right| \leq
      \|D\phi\|_{C^0}\sum_{i=0}^{k-1}Df^i(y) \\
      & \leq C\|D\phi\|_{C^0} \sum_{i=0}^{k-1} Df^i(x^\star) \leq
      C\frac{\|D\phi\|_{C_0}}{m_n(x^\star)}=
      C\frac{\|D\phi\|_{C^0}}{M_n}.
    \end{split}
  \end{displaymath}
\end{proof}

Now let us recall the classical Faa-di Bruno equation:
\begin{proposition}
  \label{pro:faadi-bruno}
  Given $g,h\in C^r(\bb{R})$, with $r\geq 1$, it holds
  \begin{displaymath}
    D^r(g\circ h)=\sum_{j=1}^r (D^jg\circ
    h)B_{r,j}(D^1h,\ldots,D^{r-j+1}h), 
  \end{displaymath}
  where $B_{r,j}$ is polynomial in $r-j+1$ variables given by
  \begin{displaymath}
    \begin{split}
      B_{r,j}&(x_1,\ldots,x_{r-j+1}) = \\
      & \sum_{(c_i)\in\Omega_{r,j}} \frac{r!}  {c_1!\ldots
        c_{r-j+1}!(1!)^{c_1}\ldots((r-j+1)!)^{c_{r-j+1}}}
      x_1^{c_1}x_2^{c_2}\ldots x_{r-j+1}^{c_{r-j+1}},
    \end{split}
  \end{displaymath}
  and where
  \begin{displaymath}
    \Omega_{r,j}:= \Big\{(c_1,\ldots,c_{r-j+1})\in\bb{N}_0^{r-j+1} :
    \sum ic_i=r,\ \sum c_i=j \Big\}. 
  \end{displaymath}
\end{proposition}

Using the formulas given in
Proposition~\ref{pro:Dg-from-DlogDg-formula} and
Proposition~\ref{pro:faadi-bruno}, together with Yoccoz' estimate of
Proposition~\ref{pro:Ds-log-Dfn}, we can extend the previous result to
higher order derivatives:
\begin{proposition}
  \label{pro:Cr-estimate}
  Let $f$ be $C^{r+1}$ with $r\geq 2$. Then there exists $C>0$
  depending only on $f$ and $r$, such that for every $\phi\in
  C^{r+1}(\T)$, every $n\geq 0$, $1\leq k\leq q_{n+1}$, and any
  $x^\star\in\bb{R}$ satisfying $m_n(x^\star)=M_n$, it holds
  \begin{equation}
    \label{eq:Cr-estimate}
    \left|D^r(\Bs^k\phi)(y)\right|\leq C\|\phi\|_{C^r} 
    \left(\frac{1}{\sqrt{M_n}}\right)^{r+1}, \quad\forall y\in
    K_n(x^\star).
  \end{equation}
\end{proposition}

\begin{proof}
  The case $r=1$ was already proved in
  Proposition~\ref{pro:C1-estimate}, so let us assume $r\geq
  2$. Applying Proposition~\ref{pro:faadi-bruno} and the estimate
  given by Corollary~\ref{cor:Dsg-from-DslogDg-estimate}, for any
  $x\in\T$ we obtain
  \begin{equation}
    \label{eq:DrBsphi-faadi}
    \begin{split}
      |D^r(\Bs^k\phi)(x)| & = \bigg|\sum_{i=0}^{k-1}D^r(\phi\circ
      f^i)(x)\bigg| \\
      & = \bigg|\sum_{i=0}^{k-1}\sum_{j=1}^r D^j\phi(f^i(x))
      B_{r,j}(Df^i(x),\ldots, D^{r-j+1}f^i(x))\bigg| \\
      & \leq \sum_{j=1}^r \|D^j\phi\|_{C^0} \sum_{i=0}^{k-1}
      |B_{r,j}(Df^i(x),\ldots, D^{r-j+1}f^i(x))| \\
      &\leq C \|\phi\|_{C^r} \sum_{j=1}^r
      \left(\frac{\sqrt{M_n}}{m_n(x)}\right)^{r-j} \sum_{i=0}^{k-1}
      (Df^i(x))^{j} \\
      &\leq C\|\phi\|_{C^r} \sum_{j=1}^{r}
      \left(\frac{\sqrt{M_n}}{m_n(x)}\right)^{r-j}
      \frac{M_n^{j-1}}{m_n(x)^j} \\
      & = C \frac{\|\phi\|_{C^r}}{m_n(x)^r} \sum_{j=1}^{r}
      \big(\sqrt{M_n}\big)^{r+j-2} \leq C \|\phi\|_{C^r}
      \frac{\big(\sqrt{M_n}\big)^{r-1}}{m_n(x)^r}.
    \end{split}
  \end{equation}

  Finally, if $y\in K_n(x^\star)$, by
  Corollary~\ref{cor:m_n-variation} we have $m_n(y)\leq C
  m_n(x^\star)=CM_n$, where $C$ is any constant bigger than
  $\exp(3\Var(\log Df))$. Thus, putting together this last estimate
  with (\ref{eq:DrBsphi-faadi}) we obtain (\ref{eq:Cr-estimate}).
\end{proof}

\section{Fibered $\bb{Z}^2$-actions and coboundaries}
\label{sec:Z2-actions}

The main purpose of this section consists in introducing the
\emph{fibered $\Z^2$-actions on $\R^2$} which shall play a central
role in our renormalization scheme. We will see that (a lift of) a
circle diffeomorphism and a cocycle naturally induce a fibered
$\Z^2$-action. Then, we shall extend the notion of coboundary to
fibered $\Z^2$-actions and in
Lemma~\ref{lem:cobound-actions-eq-cobound} we prove this new
definition indeed generalizes the previous one given in
\S\ref{sec:cocy-coub-distrib}. 

Then, in Proposition~\ref{pro:flat-cocycles-cobounds} we give a simple
but fundamental characterization of certain coboundaries which is
mainly inspired in the definition of \emph{quasi-rotations} of
Yoccoz~\cite{yoccoz-thesis}.

The space $\W^r:=\Diff+r(\bb{R})\times C^r(\bb{R})$ can be seen as a
subgroup of $\Diff+{r}(\bb{R}^2)$ defining
\begin{displaymath}
  (f,\psi): (x,y)\mapsto (f(x),y+\psi(x)),
\end{displaymath}
for each $(f,\psi)\in\W^r$.

The space of \emph{fibered $\bb{Z}^2$-actions} on $\R^2$ will be
denoted by
\begin{displaymath}
  \A^r:=\Hom(\Z^2,\W^r)\subset\Hom(\Z^2,\Diff+r(\R^2)).
\end{displaymath}

Given any $\Phi\in\A^r$ and any $(m,n)\in\Z^2$, we write
\begin{displaymath}
  \Phi(m,n)=(f^{m,n}_\Phi,\psi^{m,n}_\Phi)\in\W^r.
\end{displaymath}

Whenever the action is clear from the context, we shall just write
$(f^{m,n},\psi^{m,n})$ instead of $(f^{m,n}_\Phi,\psi^{m,n}_\Phi)$.

There are two group actions on $\A^r$ that will be used in our
renormalization scheme: the first one is the left $\W^s$-action
$\Ta\colon\W^s\times\A^r\to\A^r$ (with $0\leq s\leq r$), given by
conjugation in $\Diff+r(\R^2)$, \ie
\begin{displaymath}
  \begin{split}
    \Ta_{(g,\xi)}(\Phi) (m,n) &:= (g,\xi)\Phi(m,n)(g,\xi)^{-1} \\
    & = ( g f_\Phi^{m,n} g^{-1}, (\psi^{m,n}_\Phi + \xi f_\Phi^{m,n} -
    \xi)\circ g^{-1}),
  \end{split}
\end{displaymath}
for every $(g,\xi)\in\W^s$, $\Phi\in\A^r$ and $(m,n)\in\bb{Z}^2$. The
second one is the left $\GL(2,\bb{Z})$-action
$\U\colon\GL(2,\bb{Z})\times\A^r\to\A^r$ given by change of basis in
$\bb{Z}^2$, \ie
\begin{displaymath}
  \U_A(\Phi)(m,n):=\Phi(m^\prime,n^\prime),
\end{displaymath}
where $A\in\GL(2,\bb{Z})$ and
\begin{displaymath}
  \binom{m^\prime}{n^\prime} := A^{-1}\binom{m}{n}.
\end{displaymath}

A very simple but fundamental remark about these actions is that $\Ta$
and $\U$ commute.

Most of the time we will work on the subset $\A_0^r\subset\A^r$ given
by
\begin{displaymath}
  \A_0^r:=\big\{\Phi\in\A^r : \Fix(f_\Phi^{m,n})=\emptyset,\
  \forall (m,n)\in\bb{Z}^2\setminus\{(0,0)\}\big\}.
\end{displaymath}
Observe that this subset is invariant under the actions $\Ta$ and
$\U$.

Next, notice that each pair $(f,\phi)\in\widetilde{\Diff+r}(\T)\times
C^r(\T)\subset \W^r$ naturally induces an action
$\Gamma=\Gamma(f,\phi)\in\A^r$ given by
\begin{equation}
  \label{eq:def-induced-action}
  \Gamma : \left\{
    \begin{split}
      (1,0)&\mapsto (\tau,0), \\
      (0,1)&\mapsto (f,\phi),
    \end{split}
  \right.
\end{equation}
where $\tau$ is the translation $x\mapsto x-1$. Notice that this
action $\Gamma$ belongs to $\A^r_0$ if and only if
$\Per(\pi(f))=\emptyset$, \ie $\rho(f)$ is an irrational number.

Now, taking into account that a circle diffeomorphism and a cocycle
induce a fibered $\Z^2$-action, it is reasonable to extend the notion
of \emph{coboundary} to fibered $\Z^2$-actions: we say that
$\Phi\in\A^r_0$ is a $C^s$-\emph{coboundary}, with $0\leq s\leq r$, if
and only if there exist $(g,\xi)\in\W^s$ and $A\in\GL(2,\bb{Z})$ such
that $\Phi^\prime:=\U_A(\Ta_{(g,\xi)}\Phi)$ satisfies the following
conditions:
\begin{gather*}
  f_{\Phi^\prime}^{1,0}=\tau, \\
  \psi_{\Phi^\prime}^{1,0}=\psi_{\Phi^\prime}^{0,1}\equiv 0.
\end{gather*}

It is very easy to verify that this new notion of coboundary is
coherent with the previous one. In fact, we have the following
\begin{lemma}
  \label{lem:cobound-actions-eq-cobound}
  Let $F\in\Diff{+}{r}(\T)$ (with $0\leq r\leq\infty$) be such that
  $\rho(F)\in(\bb{R}\setminus\bb{Q})/\bb{Z}$, $\phi\in C^r(\T)$ and
  $f\in\widetilde{\Diff{+}{r}}(\T)$ be any lift of $F$. Then $\phi\in
  B(F,C^s(\T))$ if and only if the induced action $\Gamma(f,\phi)$
  given by (\ref{eq:def-induced-action}) is a $C^s$-coboundary.
\end{lemma}

\begin{proof}
  Let us start assuming that $\Gamma=\Gamma(f,\phi)$ is a
  $C^s$-coboundary. This means there exist $(g,u)\in\W^s$ and
  $A\in\GL(2,\Z)$ such that $\Xi:=\Ta_{(g,u)}(\U_A\Gamma)$ satisfies
  $\psi_\Xi^{1,0}=\psi_\Xi^{0,1}\equiv 0$ and $f_\Xi^{1,0}=\tau$. This
  implies that
  \begin{displaymath}
    \psi^{m,n}_\Xi\equiv 0, \quad\forall\: (m,n)\in\Z^2.
  \end{displaymath}

  In particular, if we write $A=\begin{pmatrix} a & c \\ b &
    d\end{pmatrix}$, we get
  \begin{gather}
    \label{eq:PsiPhi10}
    \psi_\Xi^{a,b}=\left(u\circ\tau-u\right)\circ g^{-1}\equiv 0 \\
    \label{eq:PsiPhi01}
    \psi_\Xi^{c,d}=\left(\phi+u\circ f - u\right)\circ g^{-1} \equiv
    0. 
  \end{gather}

  By (\ref{eq:PsiPhi10}), $u$ is $\bb{Z}$-periodic, and by
  (\ref{eq:PsiPhi01}) $\phi$ is a $C^s$-coboundary for $F$.

  Reciprocally, let us suppose $\phi\in B(F,C^s(\T))$. So, we can find
  $u\in C^s(\T)$ satisfying $uf-u=\phi$. Thus, writing
  $\Gamma^\prime:=\Ta_{(id,-u)}\Gamma$ we clearly have
  $f_{\Gamma^\prime}^{1,0}=f_{\Gamma}^{1,0}=\tau$ and
  $\psi_{\Gamma^\prime}^{1,0}=\psi_{\Gamma^\prime}^{0,1}\equiv
  0$. Therefore, $\Gamma$ is a $C^s$-coboundary.
\end{proof}

Our next result is a very elementary but useful characterization of
coboundaries that will turn to be our fundamental tool to construct a
coboundary as a small perturbation of a cocycle after applying our
renormalization scheme. However, first we need a very simple (and
well-known) lemma about cohomological equations on the real line:

\begin{lemma}
  \label{lem:cohomo-eq-on-R}
  Given any $f\in\Diff+r(\R)$ with $\Fix(f)=\emptyset$ and any
  $\phi\in C^r(\R)$, the cohomological equation
  \begin{displaymath}
    uf-u=\phi
  \end{displaymath}
  always admits a solution $u\in C^r(\R)$.
\end{lemma}

\begin{proof}
  Let us write $z:=f(0)$. Since $f$ is fixed-point free, we do not
  loose any generality assuming $z>0$.

  Next, let $u\colon [0,z]\to\R$ be any $C^r$ function satisfying
  \begin{displaymath}
    u\big|_{\left[0,\frac{z}{3}\right]}\equiv 0,\quad
    u\big|_{\left[\frac{2z}{3},z\right]}\equiv \phi.
  \end{displaymath}

  Now, since the interval $[0,z]$ is a fundamental domain for $f$, for
  any $x\in\R$ there is a unique $n(x)\in\bb{Z}$ such that
  $f^{-n(x)}(x)\in[0,x_1)$, and so we can extend the function $u$ to
  the whole real line by writing
  \begin{displaymath}
    u(x):=u(f^{-n(x)}(x))+\Bs^{n(x)}\phi(f^{-n(x)}(x)), \quad\forall
    x\in\R. 
  \end{displaymath}
  
  By the very definition, $u$ is a $C^r$ function and satisfies
  $uf-u=\phi$.
\end{proof}

Now we can state our characterization of coboundaries within the
context of fibered $\Z^2$-actions:
\begin{proposition}
  \label{pro:flat-cocycles-cobounds}
  Let $\Phi\in\A^r_0$ be so that there exists $x^\star\in\R$
  satisfying
  \begin{align}
    \label{eq:psi-10-vanish}
    \psi^{1,0}(x) &= 0,\quad\forall x\in [x^\star,f^{0,1}(x^\star)], \\
    \label{eq:psi-01-vanish}
    \psi^{0,1}(x) &= 0,\quad\forall x\in [x^\star,f^{1,0}(x^\star)].
  \end{align}
  Then, $\Phi$ is a $C^r$-coboundary.
\end{proposition}

\begin{proof}
  First of all, notice that we do not loose any generality supposing
  \begin{displaymath}
    f^{0,1}(x^\star)\in \big(f^{-1,0}(x^\star),f^{1,0}(x^\star)\big).
  \end{displaymath}

  In this case there are two possibilities: $f^{0,1}(x^\star)$ belongs
  either to $\big(x^\star,f^{1,0}(x^\star)\big)$, or to
  $\big(f^{-1,0}(x^\star),x^\star\big)$. Let us suppose the first case
  holds, being the second one completely analogous.

  By Lemma~\ref{lem:cohomo-eq-on-R}, we can find a function $u\in
  C^r(\R)$ such that
  \begin{equation}
    \label{eq:u-solution-psi-10}
    \psi^{1,0}=uf^{1,0}-u.
  \end{equation} 

  Notice that by (\ref{eq:psi-10-vanish}), we have
  \begin{equation}
    \label{eq:u-vanish}
    u(f^{1,0}(x))=u(x),\quad\forall x\in [x^\star,f^{0,1}(x^\star)].
  \end{equation}
  
  Now, since we are supposing that
  $[x^\star,f^{0,1}(x^\star)]\subset[x^\star,f^{1,0}(x^\star)]$, from
  (\ref{eq:u-vanish}) we can conclude there exists a unique function
  $\bar{u}\in C^r(\bb{R})$ satisfying
  \begin{equation}
    \label{eq:bar-u-def}
    \begin{split}
      \bar{u}\big|_{[x^\star,f^{1,0}(x^\star)]}&\equiv
      u\big|_{[x^\star,f^{1,0}(x^\star)]}, \\
      \bar{u}\big(f^{1,0}(x)\big)&=\bar{u}(x),\quad\forall x\in\R.
    \end{split}
  \end{equation}

  Next, if we define
  \begin{equation}
    \label{eq:bar-psi-def}
    \bar{\psi}:=\psi^{0,1}+u-uf^{0,1}\in C^r(\bb{R}),
  \end{equation}
  it can be easily shown that $\bar{\psi}$ is $f^{1,0}$-periodic. In
  fact, since $(f^{1,0},\psi^{1,0})$ and $(f^{0,1},\psi^{0,1})$
  commute, we have
  \begin{displaymath}
    \psi^{1,1}=\psi^{1,0}+\psi^{0,1}\circ
    f^{1,0}=\psi^{0,1}+\psi^{1,0}\circ f^{0,1}, 
  \end{displaymath}
  and so,
  \begin{equation}
    \label{eq:bar-psi-periodic}
    \begin{split}
      \bar{\psi}f^{1,0}&=(\psi^{0,1}+u-uf^{0,1})f^{1,0} =
      \psi^{0,1}f^{1,0}+uf^{1,0}-u + u - uf^{1,1} \\
      &= \psi^{0,1}f^{1,0} + \psi^{1,0} + u - uf^{1,1} =
      \psi^{0,1}+\psi^{1,0}f^{0,1}+ u - uf^{1,1} \\
      &= \psi^{0,1}+(uf^{1,0}-u)f^{0,1} + u - uf^{1,1} =
      \psi^{0,1}+u-uf^{0,1}\\
      & =\bar{\psi}.
    \end{split}
  \end{equation}

  Now, combining (\ref{eq:psi-01-vanish}), (\ref{eq:bar-u-def}) and
  (\ref{eq:bar-psi-periodic}) we can conclude that
  \begin{equation}
    \label{eq:bar-psi-bar-u-relation}
    \bar{\psi}=\bar{u}-\bar{u}f^{0,1}.
  \end{equation}

  Then, taking into account equations (\ref{eq:u-solution-psi-10}),
  (\ref{eq:bar-psi-def}) and (\ref{eq:bar-psi-bar-u-relation}) we can
  easily see that
  \begin{displaymath}
    \Ta_{(id,\bar{u}-u)}\Phi(\mathbf{i})=(f^{\mathbf{i}},0),
    \quad\text{for } \mathbf{i}=(1,0),\ (0,1).
  \end{displaymath}

  Finally, applying Lemma~\ref{lem:cohomo-eq-on-R} one again we can
  construct an orientation-preserving $C^r$ diffeomorphism
  $h\colon\R\to\R$ satisfying $hf^{1,0}h^{-1}=\tau$. Thus,
  \begin{displaymath}
    \Ta_{(h,\bar{u}-u)}\Phi =
    \begin{cases}
      (1,0)\mapsto (\tau,0),\\
      (0,1)\mapsto (h\circ f^{0,1}\circ h^{-1},0),
    \end{cases}
  \end{displaymath}
  and the proposition is proved.
\end{proof}

\section{Renormalization of fibered $\Z^2$-actions}
\label{sec:renorm-Z2-actions}

The main aim of this section is to introduce the renormalization scheme for
$\Z^2$-actions and to show how it can be used to construct coboundaries by  
perturbation of the original cocycle.  Indeed, the notion of coboundary (for
$\Z^2$-actions) turns out to be expressly renormalization-invariant,          
which allows us to take advantage of the smoothing effect of
renormalization.

% Our strategy to prove Theorem A can be roughly summarized as follows:
% first of all, in (\ref{eq:Gamman-def}) we define the
% $n^{th}$-\emph{renormalized $\Z^2$-actions} associated to a given
% cocyle and we will easily see that our definition of coboundary (for
% $\Z^2$-actions) is invariant under renormalization. Then, in
% lemmas~\ref{lem:Birk-sum-phi-first-vanish} and
% \ref{lem:Birk-sum-phi-second-vanish} we will construct a new cocyle,
% and using Proposition~\ref{pro:flat-cocycles-cobounds}, in
% Lemma~\ref{lem:phi-xi-coboundary} we will prove that the
% $n^{th}$-renormalized action associated to this new cocycle is indeed
% a coboundary. Finally in \S~\ref{sec:proof-main-thm}, taking advantage
% of the decay of the non-linearity of the cocycle produced by the
% renormalization scheme, we will show that this new cocycle can be
% constructed as an arbitrary small perturbation of the initial one,
% provided $n$ is well chosen.

% % and by our definition of coboundary
% (for fibered $\bb{Z}^2$-actions), this concept is invariant under
% renormalization, \ie $\Gamma(f,\phi)$ is a coboundary if and only if
% $\Gamma_n(\phi)$ is a coboundary and by
% Lemma~\ref{lem:cobound-actions-eq-cobound}, if and only if $\phi\in
% B(f,C^r(\T))$. Then, given any $n\in\bb{N}$, in the following two
% lemmas we will construct a new cocycle $\tilde\phi$ such that, as we
% will prove in Lemma~\ref{lem:phi-xi-coboundary},
% $\Gamma_n(\tilde\phi)$ satisfies the hypothesis of
% Proposition~\ref{pro:flat-cocycles-cobounds} and thus, it turn to be a
% coboundary. 

As in \S\ref{sec:estimate-cocycles}, $F$ will denote an arbitrary
minimal $C^r$ diffeomorphism of $\T$ (with $r\geq 3$),
$f\in\widetilde{\Diff+r}(\T)$ a lift of $F$, and $\phi\colon\T\to\R$
an arbitrary $C^r$ real cocycle.

To simplify the notation, we write $\alpha=\rho(f)$, and since we are
assuming $\alpha\in\R\setminus\Q$, we consider the sequences $(a_n)$,
$(\alpha_n)$, $(\beta_n)$, $(p_n)$, $(q_n)$ associated to $\alpha$
defined by (\ref{eq:an-alphan-def}), (\ref{eq:pn-def}),
(\ref{eq:qn-def}), and (\ref{eq:betan-def}).

Then, we define the matrices
\begin{equation}
  \label{eq:An-def}
  A_n=A_n(\alpha):= (-1)^{n}\left( 
    \begin{matrix}
      q_n & -p_n \\
      -q_{n-1} & p_{n-1}
    \end{matrix}
  \right), \quad\forall n\geq -1. 
\end{equation}
Notice that, by (\ref{eq:pn-qn-SL}), all the matrices $A_n$ belong to
$\GL(2,\bb{Z})$.

Now, for each $n\geq -1$ we define the $n^{th}$-\emph{renormalized}
action $\Gamma_n(\phi)$ by
\begin{equation}
  \label{eq:Gamman-def}
  \Gamma_n(\phi):= \U_{A_n}\big(\Gamma(f,\phi)\big),
  \quad\forall n\geq -1, 
\end{equation}
where, of course, $\Gamma(f,\phi)$ denotes the induced action defined
by (\ref{eq:def-induced-action}). Notice that
\begin{displaymath}
  \Gamma_n(\phi) : \left\{
    \begin{split}
      (1,0)&\mapsto (f_{n-1},\phi_{n-1})= (f^{q_{n-1}}-p_{n-1}
      ,\Bs^{q_{n-1}}_f\phi),
      \\
      (0,1)&\mapsto (f_n,\phi_n)= (f^{q_n}-p_n ,\Bs^{q_{n}}_f\phi).
    \end{split}
  \right.
\end{displaymath}

\begin{lemma}
  \label{lem:Birk-sum-phi-first-vanish}
  Let $n\geq 3$ and $x^\star\in\R$ be arbitrary. Then, there exists
  $u\in C^r(\T)$ such that the cocycle $\bar\phi:=\phi+u-uF$ satisfies
  \begin{equation}
    \label{eq:Birk-sum-phin-vanish}
    \bar\phi_{n-1}(y)=0,\quad\forall y\in J_{n-1}(x^\star).
  \end{equation}

  Moreover, there exists a real constant $C>0$ depending only on $F$
  and $r$, such that whenever $x^\star$ satisfies
  $m_{n-1}(x^\star)=M_{n-1}$, the function $u$ can be chosen
  fulfilling the following estimate:
  \begin{equation}
    \label{eq:Cr-estimate-u}
    |D^ru(y)|\leq C\left\|\phi_{n-1}\big|_{K_{n-1}(x^\star)}
    \right\|_{C^r} \Theta(\alpha,n,r)
    \left(\frac{1}{M_{n-1}}\right)^r,
  \end{equation} 
  for every $y\in J_{n-1}(x^\star)$, where
  \begin{equation}
    \label{eq:Theta-def}
    \Theta(\alpha,n,r):=\sum_{i=0}^r \left(
      \frac{\beta_{n-1}}{\beta_{n-1}-\beta_n}\right)^i.
  \end{equation}
\end{lemma}

\begin{proof}
  First, let $\zeta\colon\R\to\R$ be any auxiliary smooth function
  satisfying:
  \begin{itemize}
  \item $\zeta(x)=0$, for every $x\leq 0$;
  \item $0<\zeta(x)<1$, for every $x\in (0,1)$;
  \item $\zeta(x)=1$, for every $x\geq 1$.
  \end{itemize}

  Let $\hat x:=\pi(x^\star)\in\T$. By Lemma~\ref{lem:order-orbits} we
  know that, $\hat I_n(\hat x)\cap \hat I_{n-1}(\hat x)=\{\hat x\}$
  and $F^{q_n+q_{n-1}}(\hat x)$ belongs to the interior of $\hat
  I_{n-1}(\hat x)$. In particular, this implies that $I_n(x^\star)$
  and $f_{n-1}(I_n(x^\star))$ are disjoint, and the last interval is
  contained in $I_{n-1}(x^\star)$.

  We can assume $n$ is odd (the other case is completely analogous),
  so it holds
  \begin{displaymath}
    f_n(x^\star) < x^\star < f_{n-1}(f_n(x^\star)) =
    f_n(f_{n-1}(x^\star)) < f_{n-1}(x^\star).
  \end{displaymath}
  Notice since $\Gamma_n(\phi)\in\A_0^r$, the previous relation holds
  for every $x\in\R$.

  Now, we define $u\colon J_{n-1}(x^\star)\to \R$ by
  \begin{equation}
    \label{eq:u-first-def}
    u(y):=\zeta\left(\frac{y-x^\star}{f_n(f_{n-1}(x^\star))-x^\star}
    \right) \phi_{n-1}(f_{n-1}^{-1}(y)),
  \end{equation}
  for every $y\in
  J_{n-1}(x^\star)=[f_n(x^\star),f_{n-1}(x^\star)]$. Then, it clearly
  holds
  \begin{displaymath}
    u(f_{n-1}(y))-u(y)=\phi_{n-1}(y),
  \end{displaymath}
  whenever $y$ and $f_{n-1}(y)$ both belong to $J_{n-1}(x^\star)$, \ie
  for every $y\in I_n(x^\star)$.

  Now, we can extend our function $u$ to the whole real line $\R$ to
  get a $C^r$ $\Z$-periodic function satisfying
  \begin{equation}
    \label{eq:u-second-def}
    u(F^{q_{n-1}}(y))-u(y)=\phi_{n-1}(y),\quad\forall y\in
    \hat J_{n-1}(\hat x). 
  \end{equation}

  Notice that if we write $\bar\phi:=\phi+u-uF$, we clearly get
  \begin{displaymath}
    \bar\phi_{n-1}(y)= \Bs^{q_{n-1}}_f\bar\phi(y)=\phi_{n-1}(y)+ u(y)-
    u(f_{n-1}(y))=0,   
  \end{displaymath}
  for every $y\in J_{n-1}(x^\star)$, as desired.

  So, it remains to prove $u$ satisfies estimate
  (\ref{eq:Cr-estimate-u}). To do this, first notice that combining
  Corollary~\ref{cor:m_n-variation} and
  Proposition~\ref{pro:m_n-m_n-1-comparisson} we get
  \begin{equation}
    \label{eq:m_n-1-m_n-comparison}
    C^{-1}\frac{\beta_n}{\beta_{n-1}} \leq
    \frac{m_n(f_{n-1}(x^\star))}{m_{n-1}(x^\star)} \leq
    C\frac{\beta_n}{\beta_{n-1}},
  \end{equation}
  where $C>1$ is a constant which depends on $F$, but does not either
  on $n$ or $x^\star$.
  
  Now, to simplify the notation let us write
  \begin{displaymath}
    \ell:=|f_n(f_{n-1}(x^\star))-x^\star|.
  \end{displaymath}
  Observe that, since we are assuming $n$ is odd, we have
  $\ell=m_{n-1}(x^\star)-m_n(f_{n-1}(x^\star))$.  Hence, by
  (\ref{eq:m_n-1-m_n-comparison}) it holds
  \begin{equation}
    \label{eq:ell-estimate}
    \ell\geq \left(1- C^{-1}\frac{\beta_n}{\beta_{n-1}}\right)
    m_{n-1}(x^\star)\geq C^{-1}\frac{\beta_{n-1}-\beta_n}{\beta_{n-1}}
    M_{n-1}.
  \end{equation}

  Now, invoking Denjoy-Koksma inequality
  (Proposition~\ref{pro:denjoy-koksma}),
  Corollary~\ref{cor:Dsg-from-DslogDg-estimate}, Faa-di Bruno formula
  (Proposition~\ref{pro:faadi-bruno}),
  Proposition~\ref{pro:Cr-estimate} and estimate
  (\ref{eq:ell-estimate}), we can prove (\ref{eq:Cr-estimate-u}). In
  fact, for any $y\in J_{n-1}(x^\star)\subset K_{n-1}(x^\star)$ we
  have
  \begin{equation}
    \label{eq:Dr-u-estimate-I}
    \begin{split}
      &|D^ru(y)|=\left|\sum_{i=0}^r \binom{r}{i} D^i\zeta
        \left(\frac{y-x^\star}{\ell} \right)\ell^{-i}
        D^{r-i}(\phi_{n-1}\circ f_{n-1}^{-1})(y) \right|\\
      &\leq C \sum_{i=0}^r
      \left(\frac{\beta_{n-1}}{(\beta_{n-1}-\beta_{n}) M_{n-1}}
      \right)^i \Bigg|\sum_{j=1}^{r-i}
      (D^j\phi_{n-1})(f^{-1}_{n-1}(y)) \\
      &\qquad\qquad\qquad\qquad\qquad\qquad\quad
      B_{r-i,j}(D^1f_{n-1}^{-1},\ldots,D^{r-i-j+1}f_{n-1}^{-1})\Bigg|
      \\
      &\leq C\left\|\phi_{n-1}\big|_{K_{n-1}(x^\star)}
      \right\|_{C^r}\Theta(\alpha,n,r) \sum_{i=0}^r
      \left(\frac{1}{M_{n-1}}\right)^i \\
      & \qquad\qquad\qquad\qquad\qquad\qquad\qquad\qquad\qquad
      \sum_{j=1}^{r-i} (Df^{-1}_{n-1}(y))^j
      \left(\frac{\sqrt{M_{n-1}}}{m_{n-1}(y)}
      \right)^{r-i-j} \\
      &\leq C \left\|\phi_{n-1}\big|_{K_{n-1}(x^\star)}\right\|_{C^r}
      \Theta(\alpha,n,r) \left(\frac{1}{M_{n-1}}\right)^r,
    \end{split}
  \end{equation}
  and estimate (\ref{eq:Cr-estimate-u}) is proved.
\end{proof}

\begin{lemma}
  \label{lem:Birk-sum-phi-second-vanish}
  Let $\phi$, $x^\star$, $n$, $u$ and $\bar\phi$ be as in
  Lemma~\ref{lem:Birk-sum-phi-first-vanish}. Then there exists $\xi\in
  C^r(\T)$ such that
  \begin{gather}
    \label{eq:xi-cond-supp}
    \supp\xi\subset\hat I_{n-1}(\pi(x^\star))\cup \hat
    I_{n-1}(\pi(f_{n-1}(x^\star))),\\
    \label{eq:xi-cond-sum}
    \xi(y)+\xi(f_{n-1}(y))=\bar\phi_n(y), \quad\forall y\in
    I_{n-1}(x^\star).
  \end{gather}

  Moreover, there exists a constant $C>0$ depending only on $F$ and
  $r$ such that the function $\xi$ can constructed fulfilling the
  following estimate:
  \begin{equation}
    \label{eq:xi-Cr-estimate}
    \|\xi\|_{C^r} \leq C \left\|\bar\phi_n\big|_{I_{n-1}(x^\star)}
    \right\|_{C^r} \left(\frac{1}{M_{n-1}}\right)^r\!\!.
  \end{equation}
\end{lemma}

\begin{proof}
  As in the proof of Lemma~\ref{lem:Birk-sum-phi-first-vanish}, we
  will assume $n$ is odd, and therefore, for every $x\in\R$ it holds
  $f_n(x)<x<f_{n-1}(x)<f^2_{n-1}(x)$.

  Then, let us start defining $\xi$ on the interval
  $[x^\star,f_{n-1}^2(x^\star)]= I_{n-1}(x^\star)\cup
  I_{n-1}(f_{n-1}(x^\star))$ by writing
  \begin{equation}
    \label{eq:xi-def}
    \xi(y):=    
    \begin{cases}
      \zeta\left(\frac{y-x^\star}{f_{n-1}(x^\star)-x^\star}
      \right)\bar\phi_n(y),& \text{if } y\in
      I_{n-1}(x^\star),\\
      \left[1-\zeta\left(\frac{f_{n-1}^{-1}(y)-x^\star}
          {f_{n-1}(x^\star)-x^\star}\right)\right]
      \bar\phi_n(f_{n-1}^{-1}(y)),& \text{if } y\in
      I_{n-1}(f_{n-1}(x^\star)),
    \end{cases}
  \end{equation} 
  where $\zeta$ is the auxiliary function we used in the proof of
  Lemma~\ref{lem:Birk-sum-phi-first-vanish}.

  In this way, our function $\xi$ is clearly $C^r$ on the interiors of
  the intervals $I_{n-1}(x^\star)$ and $I_{n-1}(f_{n-1}(x^\star))$,
  and by the very properties of the auxiliary function $\zeta$, we
  have
  \begin{equation}
    \label{eq:xi-on-extrems}
    D^k\xi(x^\star)=D^k\xi\big(f_{n-1}^2(x^\star)\big)=0,
    \quad\text{for } k=0,1,\ldots,r. 
  \end{equation}

  In order to see that $\xi$ is also continuous and has continuous
  derivatives up to order $r$ at $f_{n-1}(x^\star)$, let us consider
  the fibered $\Z^2$-action $\Phi:=\Gamma_n(\bar\phi)$ (see
  (\ref{eq:Gamman-def}) for the definition of $\Gamma_n$) and notice
  that condition~(\ref{eq:Birk-sum-phin-vanish}) can be translated
  into the $\Z^2$-action language stating
  \begin{equation}
    \label{eq:Birk-sum-second-vanish-I}
    \psi^{1,0}_{\Phi}(y)=0,\quad\forall
    y\in J_{n-1}(x^\star)=[f_n(x^\star),f_{n-1}(x^\star)].  
  \end{equation}  
  
  On the other hand, since $\Phi(1,0)$ and $\Phi(0,1)$ commute in
  $\Diff+r(\R^2)$, and taking into account that $f_\Phi^{1,0}=f_{n-1}$
  and $f_\Phi^{0,1}=f_n$, we have
  \begin{equation}
    \label{eq:Birk-sum-second-vanish-II}
    \psi^{0,1}_\Phi(y)+\psi^{1,0}_\Phi(f_n(y))=
    \psi_\Phi^{1,0}(y)+\psi_\Phi^{0,1}(f_{n-1}(y)),\quad\forall
    y\in\R.
  \end{equation}
  Now, putting together equations~(\ref{eq:Birk-sum-second-vanish-I})
  and (\ref{eq:Birk-sum-second-vanish-II}), and recalling that
  $\psi_\Phi^{0,1}=\bar\phi_n$, we conclude that
  \begin{equation}
    \label{eq:Birk-sum-second-vanish-III}
    \bar\phi_n(y)=\bar\phi_n(f_{n-1}(y)), \quad\forall
    y\in I_{n-1}(x^\star),  
  \end{equation}
  
  From (\ref{eq:Birk-sum-second-vanish-III}) we can easily show that
  $\xi$ is continuous and has continuous derivatives up to order $r$
  at the point $f_{n-1}(x^\star)$.

  Hence, by this remark and (\ref{eq:xi-on-extrems}) we can affirm
  there is a unique extension of $\xi$ to the whole real line such
  that it is $C^r$, $\Z$-periodic and satisfies
  \begin{displaymath}
    \supp\xi\subset \bigcup_{k\in\Z} [x^\star+k,f^2_{n-1}(x^\star)+k],
  \end{displaymath}
  Of course, this is clearly equivalent to
  (\ref{eq:xi-cond-supp}). The condition (\ref{eq:xi-cond-sum}) is
  also satisfied by the pure construction of $\xi$.

  Next, let us prove that $\xi$ satisfies estimate
  (\ref{eq:xi-Cr-estimate}). To do this, first let $y$ be an arbitrary
  point in $I_{n-1}(x^\star)$ and notice that
  \begin{equation}
    \label{eq:xi-Cr-estimate-comput-I}
    \begin{split}
      \left|D^r\xi(y)\right|& = \left|\sum_{j=0}^{r}\binom{r}{j}
        (D^j\zeta)\bigg(\frac{y-x^\star}{f_{n-1}(x^\star)-x^\star}
        \bigg)\left(\frac{1}{M_{n-1}}\right)^j D^{r-j}
        \bar\phi_n(y)\right|\\
      &\leq C \left(\frac{1}{M_{n-1}}\right)^r
      \left\|\bar\phi_n\big|_{I_{n-1}(x^\star)}\right\|_{C^r}
    \end{split}
  \end{equation}

  Now, if $y$ denotes an arbitrary point of
  $I_{n-1}(f_{n-1}(x^\star))$ and $1\leq i\leq r$, by
  Corollary~\ref{cor:Dsg-from-DslogDg-estimate}
  and Proposition~\ref{pro:faadi-bruno} we have
  \begin{equation}
    \label{eq:xi-Cr-estimate-comput-II}
    \begin{split}
      \Bigg| D^i\Bigg[\zeta&\bigg(
      \frac{f_{n-1}^{-1}(y)-x^\star}{f_{n-1}(x^\star)-x^\star}\bigg)
      \Bigg]\Bigg| \\
      &= \Bigg|\sum_{j=1}^i (D^j\zeta)\bigg(
      \frac{f_{n-1}^{-1}(y)-x^\star}{f_{n-1}(x^\star)-x^\star}\bigg)
      \left(\frac{1}{M_{n-1}}\right)^j \\
      &\qquad\qquad\qquad\qquad\qquad\qquad
      B_{i,j}(Df_{n-1}^{-1}(y),\ldots,D^{i-j+1}f_{n-1}^{-1}(y))
      \Bigg| \\
      &\leq C \sum_{j=1}^i\left(\frac{1}{M_{n-1}}\right)^j
      (Df_{n-1}^{-1}(y))^j \left(\frac{\sqrt{M_{n-1}}}{m_{n-1}(y)}
      \right)^{i-j}\\
      &\leq C \sum_{j=1}^i\left(\frac{1}{\sqrt{M_{n-1}}}
      \right)^{i+j}\leq C \left(\frac{1}{M_{n-1}}\right)^i,
    \end{split}
  \end{equation}
  and so,
  \begin{equation}
    \label{eq:xi-Cr-estimate-comput-III}
    \begin{split}
      |D^r\xi(y)| &\leq C \sum_{i=0}^r \left|
        D^i\left[\zeta\bigg(\frac{f_{n-1}^{-1}(y)-x^\star}
          {f_{n-1}(x^\star)-x^\star}\bigg)\right]
        D^{r-i}\bar\phi_n(f_{n-1}^{-1}(y))\right| \\
      &\leq C \sum_{i=0}^r \left(\frac{1}{M_{n-1}}\right)^i
      \left\|\bar\phi_n\circ
        f_{n-1}^{-1}\big|_{I_{n-1}(f_{n-1}(x^\star))}
      \right\|_{C^{r-i}}  \\
      & \leq C \left(\frac{1}{M_{n-1}}\right)^r \left\|\bar\phi_n\circ
        f_{n-1}^{-1}\big|_{I_{n-1}(f_{n-1}(x^\star))}
      \right\|_{C^r} \\
      &= C \left(\frac{1}{M_{n-1}}\right)^r
      \left\|\bar\phi_n\big|_{I_{n-1}(x^\star)}\right\|_{C^r}.
    \end{split}
  \end{equation}
  where the last equality is consequence of
  (\ref{eq:Birk-sum-second-vanish-III}).

  Now, combining (\ref{eq:xi-Cr-estimate-comput-I}) and
  (\ref{eq:xi-Cr-estimate-comput-III}) we can easily get
  (\ref{eq:xi-Cr-estimate}).
\end{proof}

\begin{lemma}
  \label{lem:phi-xi-coboundary}
  Let $\phi$, $\bar\phi$ and $\xi$ be as in
  Lemma~\ref{lem:Birk-sum-phi-second-vanish}. Then, the cocycle
  $\tilde\phi:=\phi - \xi$ is a $C^r$-coboundary for $F$.
\end{lemma}

\begin{proof}
  Since $\phi$ and $\bar\phi$ are $C^r$-cohomologous, this is
  equivalent to show that $\bar\phi-\xi\in B(F,C^r(\T))$. To do this,
  we will show that the $\Z^2$-action $\Gamma:=\Gamma_n(\bar\phi-\xi)$
  is a $C^r$-coboundary.

  First observe that $\Gamma(1,0)=(f_{n-1},\bar\phi_{n-1}-\xi_{n-1})$.
  By (\ref{eq:Birk-sum-phin-vanish}) we know
  $\bar\phi_{n-1}\big|_{I_n(x^\star)}\equiv 0$, and if we write $\hat
  x:=\pi(x^\star)$, Lemma~\ref{lem:order-orbits} implies that for any
  $y\in \hat I_n(\hat x)$, it holds
  \begin{equation}
    \label{eq:f-i-y-not-supp-xi-I}
    F^i(y)\not\in\hat I_{n-1}(\hat x)\setminus\{\hat x\}, 
    \quad\text{for } i=0,1,\ldots, q_{n-1}-1.
  \end{equation}
  
  On the other hand, we affirm that
  \begin{equation}
    \label{eq:f-i-y-not-supp-xi-II}
    F^i(y)\not\in \hat I_{n-1}(F^{q_{n-1}}(\hat x)),
    \quad\text{for } i=0,1,\ldots, q_{n-1}-1.
  \end{equation}
  In fact, let us suppose that (\ref{eq:f-i-y-not-supp-xi-II}) does
  not hold. So, there exists $y\in \hat I_n(\hat x)$ and
  $i\in\{1,\ldots,q_{n-1}-1\}$ such that $F^i(y)\in\hat
  I_{n-1}(F^{q_{n-1}}(\hat x))$. 

  Moreover, we have
  \begin{equation}
    \label{eq:F-q_n-y-in-I__n-1}
    F^{q_{n-1}-i}(F^i(y))= F^{q_{n-1}}(y)\in F^{q_{n-1}}(\hat
    I_n(\hat x))=\hat I_n(F^{q_{n-1}}(\hat x)),
  \end{equation}
  and by Lemma~\ref{lem:order-orbits} we know that
  $F^{q_{n-1}-i}(F^{2q_{n-1}}(\hat x))\not\in\hat
  I_{n-1}(F^{q_{n-1}}(\hat x))$. In particular, this last remark and
  (\ref{eq:F-q_n-y-in-I__n-1}) imply that
  $F^{q_{n-1}-i}(F^{2q_{n-1}}(\hat x))\in I_n(F^{q_{n-1}}(\hat x))$,
  and since $F$ is topologically conjugate to the irrational rotation
  $R_\alpha$, this clearly contradicts
  (\ref{eq:q_n-as-closest-return}). Hence,
  (\ref{eq:f-i-y-not-supp-xi-II}) is proved.

  Now, by (\ref{eq:xi-cond-supp}), (\ref{eq:f-i-y-not-supp-xi-I}) and
  (\ref{eq:f-i-y-not-supp-xi-II}), we have
  $\xi_{n-1}\big|_{I_n(x^\star)}\equiv 0$, and so,
  \begin{equation}
    \label{eq:Gamma-10-vanish}
    \psi_\Gamma^{1,0}(y)=\bar\phi_{n-1}(y)=0,\quad\forall y\in
    I_n(x^\star)=[x^\star,f_\Gamma^{0,1}(x^\star)].  
  \end{equation}

  On the other hand, let $z$ be an arbitrary point in $\hat
  I_{n-1}(\hat x)$ and let us consider the set
  \begin{displaymath}
    A_z:=\left\{i\in\bb{N}_0 : F^i(z)\in \hat I_{n-1}(\hat x)\cup\hat
      I_{n-1}(F^{q_{n-1}}(\hat x)),\ i<q_n\right\}.
  \end{displaymath}
  
  Let us prove that $A_z=\{0,\: q_{n-1}\}$. To do this, first notice
  that clearly $\{0,\: q_{n-1}\}\subset A_z$. Then, consider any
  $i\in\bb{N}$ with $0<i<q_{n-1}$. Observe that by
  Lemma~\ref{lem:order-orbits} we have $F^i(z)\not\in \hat
  I_{n-1}(\hat x)$. On the other hand, if $F^i(z)$ belonged to $\hat
  I_{n-1}(F^{q_{n-1}}(\hat x))$, we would have $\{F^i(z),\:
  F^{q_{n-1}-i}(F^i(z))\}\subset \hat I_{n-1}(F^{q_{n-1}}(\hat x))$,
  which clearly contradicts Lemma~\ref{lem:order-orbits}.

  Now let $j$ be any natural number with $q_{n-1}<j<q_n$. Applying
  Lemma~\ref{lem:order-orbits} once again we know $F^j(z)\not\in \hat
  I_{n-1}(z)$. On the other hand, if $F^j(z)$ belonged to $\hat
  I_{n-1}(F^{q_{n-1}}(\hat x))$, it would hold $\{F^{q_{n-1}}(z), \:
  F^{j-q_{n-1}}(F^{q_{n-1}}(z))\}\subset \hat I_{n-1}(F^{q_{n-1}}(\hat
  x))$, which contradicts Lemma~\ref{lem:order-orbits}, too. Thus,
  $A_z=\{0,q_{n-1}\}$.

  Now, by (\ref{eq:xi-cond-sum}), it holds
  \begin{equation}
    \label{eq:Gamma-01-vanish}
    \psi_\Gamma^{0,1}(y)= \bar\phi_n(y)-\xi(y)- \xi(f_{n-1}(y))
    = 0,  
  \end{equation}
  for every $y\in I_{n-1}(x^\star)=
  [x^\star,f_\Gamma^{1,0}(x^\star)]$.

  Finally, putting together (\ref{eq:Gamma-01-vanish}),
  (\ref{eq:Gamma-10-vanish}) and
  Proposition~\ref{pro:flat-cocycles-cobounds}, we conclude $\Gamma$
  is a $C^r$-coboundary, and by
  Lemma~\ref{lem:cobound-actions-eq-cobound}, $\tilde\phi\in
  B(F,C^r(\T))$, as desired.
\end{proof}

\section{Proof of Theorem A}
\label{sec:proof-main-thm}

First of all, let us suppose $\rho(F)$ is Diophantine. Then, by
Herman-Yoccoz theorem~\cite{herman-ihes,yoccoz-ens} $F$ is smoothly
conjugate to the rigid rotation $R_{\rho(F)}$ and, by
Proposition~\ref{pro:one-dim-inv-dist}, we know $\dim
\Dis{}(R_{\rho(F)})=1$. Then, $\Dis{}(F)$ is one-dimensional, too, \ie
$\Dis{}(F)$ is spanned by the only $F$-invariant probability measure.

Therefore, from now on we can assume $F$ exhibits a Liouville rotation
number, and Theorem~A will follow as a straightforward consequence of
the following result, which can be considered as a finitary version of
it:
\begin{theorem}
  \label{thm:main-thm-finitary}
  Let $F\in\Diff+r(\T)$ (with $r\geq 5$) be such that $\rho(F)$
  satisfies the following condition: the set $\mathcal{L}(\alpha,r/2)$
  given by (\ref{eq:L-alpa-delta}) contains infinitely many elements
  for some, and hence any, $\alpha\in\pi^{-1}(\rho(F))\subset\R$ (e.g
  when $\rho(F)$ is Liouville).

  Let $k:=\left\lfloor\frac{r-5}{6}\right\rfloor$ and $\phi\in
  C^k(\T)$ be such that
  \begin{displaymath}
    \int_{\T}\phi\dd\mu=0,
  \end{displaymath}
  where $\mu$ is the only $F$-invariant probability measure.

  Then, given any $\epsilon>0$, there exists $\tilde\phi\in C^k(\T)$
  such that $\tilde\phi\in B(F,C^k(\T))$ and
  \begin{equation}
    \label{eq:xi-small-Ck}
    \|\tilde\phi-\phi\|_{C^k}\leq\epsilon.
  \end{equation}
\end{theorem}

Notice that, by Proposition~\ref{pro:inv-dist-vs-cobound}, the
conclusion of this theorem can be briefly summarized saying that
$\Dis{k}(F)=\R\mu$.

\begin{proof}[Proof of Theorem~\ref{thm:main-thm-finitary}] 
  First, let us fix a lift $f\in\widetilde{\Diff+r}(\T)$ of $F$ and
  then we can suppose $\alpha:=\rho(f)$. Let $(q_n)$ and $(\beta_n)$ the
  sequences given by (\ref{eq:qn-def}) and (\ref{eq:betan-def})
  associated to the continued fraction expansion of $\alpha$.

  By our arithmetical hypothesis $\Lie(\alpha,r/2)$ is a infinite set
  and so, we can find $n\in\Lie(\alpha,r/2)$ with $n\geq n_0$, where
  $n_0$ is the natural number given by
  Proposition~\ref{pro:Ds-log-Dfn+1}. Let $x^\star\in\R$ be any point
  such that $m_{n-1}(x^\star)=M_{n-1}$.

  Now, by combining Proposition~\ref{pro:m_n-m_n-1-comparisson} and
  Proposition~\ref{pro:Ds-log-Dfn+1}, for any $0\leq s\leq r-2$ and
  any $y\in K_{n-1}(x^\star)$ we have
  \begin{equation}
    \label{eq:Ds-logDfn-Liouville}
    \begin{split}
      \big|D^s\log &Df_n(y) \big| \leq C
      \frac{\sqrt{M_{n-1}}}{(m_{n-1}(x^\star))^s}
      \left(\Big(\sqrt{M_{n-1}}\Big)^{r-2} +
        \frac{m_n(x^\star)}{m_{n-1}(x^\star)}\right) \\
      & \leq C \left( (M_{n-1})^{\frac{r-1}{2}-s} +
        \frac{\beta_n}{\beta_{n-1}}\big(M_{n-1}\big)^{\frac{1}{2}-s}
      \right) \\
      & \leq C\left((M_{n-1})^{\frac{r-1}{2}-s} +
        \beta_{n-1}^{\frac{r}{2}-1}\big(M_{n-1}\big)^{\frac{1}{2}-s}
      \right) \\
      & \leq C \big(M_{n-1}\big)^{\frac{r-1}{2}-s},
    \end{split}
  \end{equation}
  where the last inequality is consequence of the fact that
  \begin{displaymath}
    \beta_{n-1}=\int_\T m_{n-1}(t) \dd\mu(t) \leq M_{n-1}. 
  \end{displaymath}

  Then, if $P_s\in\Z[X_1,\ldots,X_s]$ denotes the polynomial of degree
  $s$ given by Proposition~\ref{pro:Dg-from-DlogDg-formula}, applying
  (\ref{eq:pseudo-homogen-polynoms-Pr}) and
  (\ref{eq:Ds-logDfn-Liouville}) we get
  \begin{equation}
    \label{eq:Ds+1-Dfn-Liouville}
    \begin{split}
      \left|D^{s+1}f_n(y)\right|& = \left|P_s(D\log
        Df_n(y),\ldots, D^s\log Df_n(y))Df_n(y)\right| \\
      & \leq C P_s\Big(\big(M_{n-1}\big)^{\frac{r-1}{2}-1},
      \big(M_{n-1}\big)^{\frac{r-1}{2}-2},\ldots,
      \big(M_{n-1}\big)^{\frac{r-1}{2}-s} \Big) \\
      & \leq C \left(\frac{1}{M_{n-1}}\right)^s
      P_s\Big(\big(M_{n-1}\big)^{\frac{r-1}{2}},\ldots,
      \big(M_{n-1}\big)^{\frac{r-1}{2}}\Big) \\
      & \leq C \big(M_{n-1}\big)^{s\frac{r-1}{2}-s} = C
      \Big(\sqrt{M_{n-1}}\Big)^{rs-3s},
    \end{split}
  \end{equation}
  for every $y\in K_{n-1}(x^\star)$.

  Now, let $u\in C^r(\T)$ be the function we constructed in
  Lemma~\ref{lem:Birk-sum-phi-first-vanish}. We affirm that we can
  find a constant $C>0$, depending only on $f$ and $r$, such that
  \begin{equation}
    \label{eq:Cs-estimate-u-uf}
    \left\|uf_n-u\big|_{I_{n-1}(x^\star)}\right\|_{C^s} \leq
    C\|\phi\|_{C^{s+1}} \Big(\sqrt{M_{n-1}}\Big)^{r-3s-4}
  \end{equation}
  To prove this, first notice that for any $0\leq s\leq r-2$ and $y\in
  I_{n-1}(x^\star)$, we have
  \begin{equation}
    \label{eq:Ds-u-Ds-u-fn}
    \begin{split}
      \big|D^su&\big(f_n(y)\big)-D^su(y)\big|\leq \int_{y}^{f_n(y)}
      \left|D^{s+1}u(t)\right|\dd\Leb(t) \\
      &\leq m_n(y)\left\|D^{s+1}u\big|_{I_{n-1}(x^\star)}
      \right\|_{C^0} \\
      &\leq C\frac{\beta_n}{\beta_{n-1}}
      M_{n-1}\left\|\phi_{n-1}\big|_{K_{n-1}(x^\star)}\right\|_{C^{s+1}}
      \Theta(\alpha,n,s+1)\left(\frac{1}{M_{n-1}}\right)^{s+1}\\
      &\leq C \beta_{n-1}^{\frac{r}{2}-1}\|\phi\|_{C^{s+1}}
      \left(\frac{1}{\sqrt{M_{n-1}}}\right)^{s+2}
      \left(\frac{1}{1-\beta_{n-1}^{r/2-1}}\right)^{s+1}
      \left(\frac{1}{M_{n-1}}\right)^s \\
      &\leq C \|\phi\|_{C^{s+1}}\big(M_{n-1}\big)^{\frac{r}{2}-1}
      \left(\frac{1}{M_{n-1}}\right)^{\frac{3s+2}{2}}= C
      \|\phi\|_{C^{s+1}} \Big(\sqrt{M_{n-1}}\Big)^{r-3s-4}.
    \end{split}
  \end{equation}
  Observe that (\ref{eq:Ds-u-Ds-u-fn}) is indeed the proof of
  (\ref{eq:Cs-estimate-u-uf}) for the particular case $s=0$.

  In the other cases, that is when $s\geq 1$, we can use Faa-di Bruno
  equation and estimates (\ref{eq:Ds-logDfn-Liouville}),
  (\ref{eq:Ds+1-Dfn-Liouville}) and (\ref{eq:Ds-u-Ds-u-fn}) to prove
  (\ref{eq:Cs-estimate-u-uf}):
  \begin{equation}
    \label{eq:Ds(u-ufn)-estimate}
    \begin{split}
      &\Big|D^s(uf_n-u)(y)\Big|\leq \Big|D^su(f_n(y))Df_n(y)-
      D^su(y)\Big| \\
      &\qquad\qquad\qquad\qquad+ \left|\sum_{j=1}^{s-1} D^ju(f_n(y))
        B_{s,j}(Df_n(y),\ldots,D^{s-j+1}f_n(y))\right|\\
      &\leq Df_n(y)\left|D^su(f_n(y))-D^su(y)\right| + |Df_n(y)-1|
      |D^su(y)| \\
      & \qquad\qquad+ C \sum_{j=1}^{s-1}
      \left\|\phi_{n-1}\big|_{K_{n-1}(x^\star)}\right\|_{C^j}
      \Theta(\alpha,n,j)\left(\frac{1}{M_{n-1}}\right)^j \\
      &\qquad\qquad\qquad\qquad
      B_{s,j}\Big(1,\Big(\sqrt{M_{n-1}}\Big)^{r-3},\ldots,
      \Big(\sqrt{M_{n-1}}\Big)^{(r-3)(s-j)}\Big) \\
      &\leq C \left(\|\phi\|_{C^{s+1}} M_{n-1}^{\frac{r-3s-4}{2}} +
        M_{n-1}^{\frac{r-1}{2}}
        \left\|\phi_{n-1}\big|_{K_{n-1}(x^\star)}\right\|_{C^s}
        M_{n-1}^{-s}\right) \\
      & \qquad\qquad+ C\sum_{j=1}^{s-1}\|\phi\|_{C^j}
      \left(\frac{1}{\sqrt{M_{n-1}}}\right)^{j+1}
      \left(\frac{1}{M_{n-1}}\right)^j
      \Big(\sqrt{M_{n-1}}\Big)^{(r-3)(s-j)} \\
      &\leq C \|\phi\|_{C^{s+1}}
      \left[\Big(\sqrt{M_{n-1}}\Big)^{r-3s-4} +
        \Big(\sqrt{M_{n-1}}\Big)^{r-3s-1}\right] \\
      &\leq C\|\phi\|_{C^{s+1}} \Big(\sqrt{M_{n-1}}\Big)^{r-3s-4},
    \end{split}
  \end{equation}
  and (\ref{eq:Cs-estimate-u-uf}) is proved.

  Now, let us consider the cocycle $\bar\phi$ as defined in
  Lemma~\ref{lem:Birk-sum-phi-first-vanish}, \ie given by
  $\bar\phi:=\phi+u-uf$. Notice that whenever $0\leq s\leq
  \frac{2r-3}{3}$ , combing Proposition~\ref{pro:Cr-estimate} and
  (\ref{eq:Cs-estimate-u-uf}) we get
  \begin{equation}
    \label{eq:Cs-estimate-bar-phin}
    \left|D^s\bar\phi_n(y)\right|\leq \left|D^s\phi_n(y)\right| +
    \left|D^s(u- uf_n)(y)\right|\leq
    C\|\phi\|_{C^{s+1}}\left(\frac{1}{\sqrt{M_{n-1}}}\right)^{s+1}, 
  \end{equation}
  for every $y\in I_{n-1}(x^\star)$.

  On the other hand, remember that in the middle of the proof of
  Lemma~\ref{lem:Birk-sum-phi-second-vanish} we show that $\bar\phi$
  satisfies (\ref{eq:Birk-sum-second-vanish-III}), that is
  \begin{equation}
    \label{eq:Birk-sum-second-vanish-III-reloaded}
    \bar\phi_n(y)=\bar\phi_n(f_{n-1}(y)),\qquad\forall y\in
    I_{n-1}(x^\star).
  \end{equation}

  This implies there exists a unique function $\gamma\in C^r(\R)$
  which coincides with $\bar\phi_n$ on $I_{n-1}(x^\star)$ and is
  $f_{n-1}$-invariant on the whole real line. To estimate the
  $C^r$-norm of $\gamma$ first observe that, by the definition of
  $\gamma$ and estimate (\ref{eq:Cs-estimate-bar-phin}), it holds
  \begin{displaymath}
    \left\|\gamma\big|_{I_{n-1}(x^\star)}\right\|_{C^s} =
    \left\|\bar\phi_n\big|_{I_{n-1}(x^\star)}\right\|_{C^s} \leq
    C\|\phi\|_{C^{s+1}}\left(\frac{1}{\sqrt{M_{n-1}}}\right)^{s+1}.
  \end{displaymath} 

  On the other hand, applying estimate (\ref{eq:Cs-estimate-bar-phin})
  and recalling $\gamma=\gamma\circ f_{n-1}$, for any $y\in
  I_{n-1}(f_{n-1}(x^\star))$ we get
  \begin{equation}
    \label{eq:Cs-estimate-gamma-f-I-n-1}
    \begin{split}
      |D^s\gamma(y)|&=|D^s(\gamma\circ
      f_{n-1}^{-1})(y)|=|D^s(\bar\phi_n\circ f_{n-1}^{-1})(y)| \\
      &= \bigg|\sum_{i=1}^s D^i\bar\phi_n(f_{n-1}^{-1}(y))
      B_{s,i}(Df_{n-1}^{-1}(y),\ldots,D^{s-i+1}f_{n-1}^{-1}(y)) \bigg|
      \\
      & \leq \sum_{i=1}^s
      \Big\|\bar\phi\big|_{I_{n-1}(x^\star)}\Big\|_{C^i} \big|
      B_{s,i}(Df_{n-1}^{-1}(y),\ldots,D^{s-i+1}f_{n-1}^{-1}(y))
      \big| \\
      &\leq C\|\phi\|_{C^{s+1}} \sum_{i=1}^s
      \left(\frac{1}{\sqrt{M_{n-1}}}\right)^{i+1} (Df_{n-1}^{-1}
      (y))^i
      \left(\frac{1}{\sqrt{M_{n-1}}} \right)^{s-i} \\
      &= C\|\phi\|_{C^{s+1}}
      \left(\frac{1}{\sqrt{M_{n-1}}}\right)^{s+1}.
    \end{split}
  \end{equation}
  
  Repeating this argument we can show that, given any $k\in\bb{N}$,
  there exists a constant $C_k>0$ depending only on $f$, $r$ and $k$
  such that
  \begin{equation}
    \label{eq:Cs-estimate-gamma-fk-I-n-1}
    |D^s\gamma(y)|\leq C_k\|\phi\|_{C^{s+1}}
    \left(\frac{1}{\sqrt{M_{n-1}}}\right)^{s+1},\quad\forall y\in
    I_{n-1}(f_{n-1}^k(x^\star)).
  \end{equation} 

  Now, returning to (\ref{eq:Birk-sum-second-vanish-III-reloaded}) we
  can affirm there exists $x_1\in I_{n-1}(x^\star)$ satisfying
  \begin{displaymath}
    D\gamma(x_1)=D\bar\phi_n(x_1)=0.
  \end{displaymath}
  Moreover, since $\gamma$ is $f_{n-1}$-periodic and $f_{n-1}$ is a
  diffeomorphism, $x_1^k:=f_{n-1}^k(x_1)$ is a critical point of
  $\gamma$, for every $k\in\Z$.

  This implies that, for each integer $k$, we can find a point
  $x_2^k\in [f_{n-1}^k(x_1),f_{n-1}^{k+1}(x_1)]$ such that
  $D^2\gamma(x_2^k)=0$, and inductively, we can define the sequence
  (of sequences) of points $(x_s^k)_{1\leq s\leq r,k\in\Z}$ that
  satisfies
  \begin{equation}
    \label{eq:x-s-k-definition}
    D^s\gamma(x_s^k)=0, \quad\text{and}\quad x_{s+1}^k\in
    [x_s^k,x_s^{k+1}]\subset\R,
  \end{equation}
  for every $1\leq s\leq r$ and every $k\in\Z$. Now, if we write
  \begin{displaymath}
    \mathcal{I}_s:=I_{n-1}(x^\star)\cup I_{n-1}(f_{n-1}(x^\star))\cup
    \cdots \cup I_{n-1}(f^{s-1}_{n-1}(x^\star)),  
  \end{displaymath}
  one can easily check that
  \begin{equation}
    \label{eq:x-s-0-and-I-s-prop}
    x_s^0\in \mathcal{I}_s \quad\text{and}\quad
    \Leb(\mathcal{I}_s)\leq sM_{n-1}. 
  \end{equation}

  Now, estimates (\ref{eq:Cs-estimate-gamma-fk-I-n-1}),
  (\ref{eq:x-s-k-definition}) and (\ref{eq:x-s-0-and-I-s-prop}) can be
  used to improved our estimate (\ref{eq:Cs-estimate-bar-phin}). In
  fact, if $0\leq s\leq \frac{2r-3}{3}$, one can easily check that
  \begin{equation}
    \label{eq:Cr-1-bar-phi-n-estimate}
    \begin{split}
      \left|D^{s-1}\gamma(y)\right| &\leq \int_{x_s^0}^y
      \left|D^s\gamma(z)\dd z \right| \leq C\|\phi\|_{C^{s+1}}
      \left(\frac{1}{\sqrt{M_{n-1}}}\right)^{s+1}s M_{n-1} \\
      & = C\|\phi\|_{C^{s+1}} M_{n-1}^{1-\frac{s+1}{2}}, \quad\forall
      y\in\mathcal{I}_s.
    \end{split}
  \end{equation}

  Iterating this procedure of integration from the appropriate point
  $x_s^0$ we get
  \begin{equation}
    \label{eq:Cr-j-bar-phi-n-estimate}
    \begin{split}
      \left|D^{\bar s-j}\bar\phi_n(y)\right|&=\left|D^{\bar
          s-j}\gamma(y)\right| \leq C\|\phi\|_{C^{\bar s+1}}
      \left(\frac{1}{\sqrt{M_{n-1}}}\right)^{\bar s+1}M_{n-1}^j \\
      &= C\|\phi\|_{C^{\bar s+1}} M_{n-1}^{j-\frac{\bar s+1}{2}},
      \quad\forall y\in I_{n-1}(x^\star),
    \end{split}
  \end{equation}
  and each $j\in\{0,1,\ldots, \bar s-1\}$, where $\bar s:=\left\lfloor
    \frac{2r-3}{3}\right\rfloor$.

  Now recall we are assuming $\phi$ has zero average with respect to
  $\mu$. Since $\bar\phi$ is is cohomologous to $\phi$, the same holds
  for $\bar\phi$. Therefore, there must exist a point $x_0\in
  I_{n-1}(x^\star)$ such that $\bar\phi_n(x_0)=0$. In particular,
  estimate (\ref{eq:Cr-j-bar-phi-n-estimate}) also holds for $j=\bar
  s$. Then, recalling the number $k$ is equal to
  $\left\lfloor\frac{r-5}{6}\right\rfloor$, we have
  \begin{equation}
    \label{eq:Ck-estimate-bar-phi-n}
    \left\|\bar\phi_n\big|_{I_{n-1}(x^\star)} \right\|_{C^k}\leq
    C\|\phi\|_{C^{\bar s+1}}  
    M_{n-1}^{\bar s-k-\frac{\bar s+1}{2}}\leq C\|\phi\|_{C^r}
    M_{n-1}^{\frac{r-3}{6}}.  
  \end{equation}

  Then, we apply Lemma~\ref{lem:Birk-sum-phi-second-vanish} to
  construct the function $\xi\in C^r(\T)$ and putting together
  estimates (\ref{eq:xi-Cr-estimate}) and
  (\ref{eq:Ck-estimate-bar-phi-n}) we obtain
  \begin{equation}
    \label{eq:final-Ck-estimate-xi}
    \|\xi\|_{C^k}\leq CM_{n-1}^{-k}
    \left\|\bar\phi_n\big|_{I_{n-1}(x^\star)} \right\|_{C^k}\leq C
    \|\phi\|_{C^r}M_{n-1}^{\frac{r-3}{6}-k}\leq
    C\|\phi\|_{C^r}M_{n-1}^{\frac{1}{3}}, 
  \end{equation}

  Taking into account that $C$ is a real constant which only depends
  on $F$ and $r$, that, by the minimality, $M_m\to 0$ as $m\to+\infty$
  and that $\Lie(\alpha,r/2)$ has infinitely many elements, we
  conclude we can choose $n\in\mathcal{L}(\alpha,r/2)$ big enough such
  that $C\|\phi\|_{C^r}M_{n-1}^{1/3}\leq \epsilon$.

  Finally, by Lemma~\ref{lem:phi-xi-coboundary} the cocycle
  $\tilde\phi:=\phi-\xi$ is a $C^r$-coboundary for $F$, and by the
  previous remark we have
  $\|\tilde\phi-\phi\|_{C^k}=\|\xi\|_{C^k}<\epsilon$, as desired.
\end{proof}

\bibliographystyle{amsalpha} \bibliography{invdist}

\vspace{.8cm}

\textsc{CNRS 7586, Institut de Math\'ematiques de Jussieu, 175 rue du
  Chevaleret, 75013, Paris -- France.}

\textit{E-mail addres:} \texttt{artur@math.sunysb.edu}

\vspace{.8cm}

\textsc{Instituto de Matem\'atica, Universidade Federal
  Fluminense. Rua M\'ario Santo Braga, s/n. Niter\'oi, RJ -- Brazil}

\textit{E-mail address:} \texttt{alejandro@mat.uff.br}

\end{document}